\documentclass{amsart}
\usepackage[l2tabu, orthodox]{nag}
\usepackage{amssymb, amsthm, enumerate, mathtools, mathrsfs, diagrams}
\usepackage[pagebackref]{hyperref}

\newtheorem{definition}{Definition}[subsection]

\newtheorem{theorem}[definition]{Theorem}

\newtheorem{lemma}[definition]{Lemma}
\newtheorem*{theorem*}{Theorem}

\newtheorem{corollary}[definition]{Corollary}
\newtheorem{proposition}[definition]{Proposition}
\newtheorem{assumption}[definition]{Assumption}

\theoremstyle{remark}
\newtheorem{remark}[definition]{Remark}
\newtheorem{notation}[definition]{Notation}

\renewcommand{\AA}{\mathbf{A}}
\newcommand{\CC}{\mathbf{C}}
\newcommand{\QQ}{\mathbf{Q}}

\newcommand{\TT}{\mathbf{T}}
\newcommand{\ZZ}{\mathbf{Z}}

\newcommand{\fa}{\mathfrak{a}}

\newcommand{\ff}{\mathfrak{f}}

\newcommand{\fn}{\mathfrak{n}}
\newcommand{\fp}{\mathfrak{p}}
\newcommand{\fP}{\mathfrak{P}}
\newcommand{\fpb}{\overline{\fp}}
\newcommand{\fq}{\mathfrak{q}}

\newcommand{\fl}{\mathfrak{l}}
\newcommand{\fr}{\mathfrak{r}}
\newcommand{\fO}{\mathfrak{O}}

\newcommand{\cC}{\mathcal{C}}
\newcommand{\cO}{\mathcal{O}}
\newcommand{\cZ}{\mathcal{Z}}
\newcommand{\cI}{\mathcal{I}}

\newcommand{\bfz}{\mathbf{z}}
\newcommand{\bfc}{\mathbf{c}}

\newcommand{\bfg}{\pmb{g}}

\newcommand{\ab}{\mathrm{ab}}
\newcommand{\et}{\text{\normalfont{\'et}}} 

\newcommand{\mot}{\mathrm{mot}}
\newcommand{\ur}{\mathrm{ur}}

\newcommand{\Zp}{\ZZ_p}
\newcommand{\Qp}{\QQ_p}
\newcommand{\Qb}{\overline{\QQ}}
\newcommand{\Qpb}{\Qb_p}

\newcommand{\into}{\hookrightarrow}
\newcommand{\onto}{\twoheadrightarrow}

\newcommand{\DdR}{\operatorname{\mathbf{D}}_{\mathrm{dR}}}

\newcommand{\stbt}[4]{\left(\begin{smallmatrix}#1 & #2 \\ #3 & #4\end{smallmatrix}\right)}
\newcommand{\quot}[1]{\text{``}{#1}\text{''}}

\DeclareMathOperator{\Fil}{Fil}
\DeclareMathOperator{\GL}{GL}
\DeclareMathOperator{\SL}{SL}
\DeclareMathOperator{\Gal}{Gal}
\DeclareMathOperator{\Ind}{Ind}
\DeclareMathOperator{\End}{End}

\DeclareMathOperator{\Aut}{Aut}
\DeclareMathOperator{\Tr}{Tr}
\DeclareMathOperator{\CH}{CH}
\DeclareMathOperator{\Hom}{Hom}
\newcommand{\disc}{\operatorname{disc}(K/\QQ)}
\DeclareMathOperator{\pr}{pr}
\DeclareMathOperator{\norm}{norm}
\DeclareMathOperator{\ord}{ord}
\DeclareMathOperator{\loc}{loc}
\DeclareMathOperator{\ind}{ind}

\DeclareMathOperator{\Sel}{Sel}
\DeclareMathOperator{\Spec}{Spec}
\DeclareMathOperator{\Hyp}{Hyp}

\begin{document}
 \title{Euler systems for modular forms over imaginary quadratic fields}

 \author{Antonio Lei}
 \address[Lei]{D\'epartement de math\'ematiques et de statistique,
 Universit\'e Laval, Pavillon Alexandre-Vachon,
 1045 avenue de la M\'edecine,
 Qu\'ebec, QC, Canada G1V 0A6}
 \email{antonio.lei@mat.ulaval.ca}

 \author{David Loeffler}
 \address[Loeffler]{Mathematics Institute\\
 Zeeman Building, University of Warwick\\
 Coventry CV4 7AL, UK}
 \email{d.a.loeffler@warwick.ac.uk}

 \author{Sarah Livia Zerbes}
 \address[Zerbes]{Department of Mathematics \\
 University College London\\
 Gower Street, London WC1E 6BT, UK}
 \email{s.zerbes@ucl.ac.uk}

 \thanks{The authors' research is supported by the following grants: Royal Society University Research Fellowship (Loeffler); EPSRC First Grant EP/J018716/1 (Zerbes).}


 \begin{abstract}
  We construct an Euler system attached to a weight 2 modular form twisted by a Gr\"ossencharacter of an imaginary quadratic field $K$, and apply this to bounding Selmer groups.
 \end{abstract}

 \maketitle

 \setcounter{tocdepth}{1}
 \tableofcontents

 \section{Introduction}

  \subsection{The main result}

  The main result of this paper is as follows. Let $f$ be an elliptic modular newform of weight 2 that is not of CM type, and $p \ge 5$ a prime not dividing the level of $f$. Let $K$ be an imaginary quadratic field in which $p$ is split, $L$ a sufficiently large number field (containing $K$ and the Fourier coefficients of $f$), and $\fP$ a prime of $L$ above $p$ at which $f$ is ordinary (i.e.~$v_{\fP}(a_p(f)) = 0$).

  Then one can define two $p$-adic $L$-functions $L_{\fP}(f/K, \Sigma^{(1)})$ and $L_{\fP}(f/K, \Sigma^{(2)})$ (\S \ref{sect:defLfcn}), which are functions on the space of characters of the ray class group of $K$ modulo $\ff p^\infty$ (for some integral ideal $\ff$ coprime to $p$ and the level of $f$). In particular, one can evaluate these $p$-adic $L$-functions at any algebraic Gr\"ossencharacter of $K$ of conductor dividing $\ff p^\infty$.

  \begin{theorem*}[{Theorem \ref{thm:finiteBKSel}}]
   Let $\psi$ be a Gr\"ossencharacter of conductor dividing $\ff$ and infinity-type $(-1, 0)$. Suppose that the $L$-values $L_{\fP}(f/K, \Sigma^{(1)})(\psi)$ and $L_{\fP}(f/K, \Sigma^{(2)})(\psi)$ are not both zero, and the following technical conditions hold:
   \begin{itemize}
    \item $\alpha \psi(\fpb) \not\equiv 1 \bmod \fP$ and $\beta \psi(\fpb) \ne p$, where $\alpha$ and $\beta$ are the unit and non-unit roots of the Hecke polynomial of $f$ at $p$, and $\fp$ is the prime of $K$ below $\fP$;
    \item $\alpha \psi(\fp) / p \notin \mu_{p^\infty}$;
    \item $p$ is unramified in the coefficient field $L$.
   \end{itemize}
   Then the Bloch--Kato Selmer group of the $\Gal(\overline{K} / K)$-representation $V_{L_{\fP}}(f)(\psi)(1)$ is finite.
  \end{theorem*}

  Under some slightly stronger technical assumptions, we can extend this result as follows. We define in \S \ref{sect:criticalSel} two groups $\Sel(K, T^\vee(1), \Sigma^{(i)})$, for $i = 1, 2$, which we call ``critical Selmer groups'', each of which contains the Bloch--Kato Selmer group.  These critical Selmer groups can be viewed as ``analytic continuations'' of the Bloch--Kato Selmer groups attached to twists of $f$ which are critical in the sense of Deligne. We show that for each $i$, if the value $L_{\fP}(f/K, \Sigma^{(i)})(\psi)$ is non-zero, then $\Sel(K, T^\vee(1), \Sigma^{(i)})$ is finite (Theorem \ref{thm:finitecriticalSel}). Morever, we obtain explicit bounds on the orders of these Selmer groups in terms of the valuations of the corresponding $L$-values.

 \subsection{Relation to our earlier work}

  In \cite{leiloefflerzerbes13} we proved a result on the finiteness of the strict Selmer group over $\QQ$ attached to the Rankin--Selberg convolution of two modular forms $f, g$, under rather strong ``large image'' assumptions on $f$ and $g$. The proof of this result relied on an Euler system constructed from generalizations of the Beilinson--Flach classes in $K_1$ of products of modular curves.

  The Selmer groups we study in the present paper can also be interpreted in terms of Rankin--Selberg convolutions: they are the Selmer groups over $\QQ$ of the convolution of $f$ with the theta-series modular form arising from $\psi$. However, the main theorem of \cite{leiloefflerzerbes13} does not apply in this situation, as the Galois representation attached to a theta series will be of dihedral type, and thus does not have large image. So we shall extend the Euler system by constructing additional cohomology classes, corresponding to abelian extensions of $K$ which are not abelian over $\QQ$. In order to construct these classes, we use maps similar to those appearing in the Taylor--Wiles method in modularity lifting theory, allowing us to patch together cohomology groups arising from modular curves of different levels. This gives an Euler system over $K$ for the Galois representation of $f$ twisted by $\psi$ (Theorem \ref{thm:eulersystem}); and applying the ``Euler system machine'' of \cite{rubin00} over
$K$, rather than over $\QQ$, then gives a bound for the strict Selmer group when the corresponding $p$-adic $L$-value is non-zero (Theorems \ref{thm:finiteSelK} and \ref{thm:boundedSelK}).

  The second new ingredient in this paper is that we bound the Bloch--Kato Selmer group, rather than the (generally smaller) strict Selmer group. In order to obtain this stronger result, we make use of an extra property of our Euler system classes: that they are in the Bloch--Kato $H^1_f$ subspaces at the primes above $p$ (which is a non-trivial condition since the Hodge--Tate weights of our representation are not all $\ge 1$). We show in this paper how to modify the Euler system machine to take into account this additional local input; this allows us to bound the Bloch--Kato Selmer group (Theorem \ref{thm:finiteBKSel}), and the two slightly larger groups we call ``critical Selmer groups''.

  \subsection{Relations to other work}

   A number of previous works (\cite{BD05}, \cite{howard}, \cite{castella14}) have explored a rather different kind of Euler system attached to modular forms over an imaginary quadratic field, arising from Heegner points or Heegner cycles, and applied these to prove bounds for Selmer groups. Our approach is somewhat different to these works, since the geometric input in our work comes from classes in $K_1$ of modular surfaces, rather than $K_0$; in particular, the existence and non-triviality of our classes is not reliant on any root number phenomena, so we can bound Selmer groups attached to twists of $f$ which are not necessarily self-dual.

   The existence of these two approaches raises the natural question of whether the specialization of our Euler system to the self-dual twists coincides with the ``big Heegner point'' Euler system of Howard and Castella. Sadly the methods of the present paper do not provide enough information about these specializations to answer this question. We hope to return to this matter in future work.

   A third approach to the study of Selmer groups for modular forms over imaginary quadratic fields is to be found in the work of Skinner and Urban \cite{skinnerurban14}. Their approach relies on establishing a \emph{lower} bound on the size of the Selmer group, and then using the upper bounds given by Kato's Euler system over $\QQ$ to show that this bound is sharp. This second step in their strategy is only applicable when the Gr\"ossencharacter $\psi$ is congruent modulo $p$ to a character factoring through the norm map to $\QQ$. However, in order to apply our methods we need precisely the opposite assumption -- our methods require that $\psi$ is \emph{not} congruent to any such character, since this would violate the ``non-Eisenstein'' condition of Definition \ref{def:nonEis}. Thus our upper bounds for the Selmer group are complementary to the results of \cite{skinnerurban14}\footnote{The method of \cite{skinnerurban14} gives lower bounds on the Selmer group in much greater generality, and it would be an
interesting project to compare these lower bounds with the upper bounds proved in this paper; we hope to investigate this in a future work.}.

  \subsection*{Acknowledgements}

   Although this paper has emerged as a follow-up to our previous paper \cite{leiloefflerzerbes13}, the CM setting considered here was the original motivation for our study of Beilinson--Flach classes, based on the conjectures about Euler systems advanced by the second and third authors in \cite{loefflerzerbes11}. We are very grateful to Massimo Bertolini, Henri Darmon, and Victor Rotger for the suggestion (made to one of us at the 2011 Durham conference) that the Beilinson--Flach classes introduced by them in \cite{BDR12} could perhaps be used in proving these conjectures, and encouraging us to pursue this idea. We would also like to express our gratitude for all they have done to support our work in this area since, and for the continuing inspiration offered by their own work in the field.

   The idea used in this paper of patching together an Euler system from classes in the motivic cohomology of many Shimura varieties, rather than just one, was inspired by an earlier paper of Bertolini and Darmon on the anticyclotomic Iwasawa theory of modular forms \cite{BD05}. We are grateful to Henri Darmon for bringing this paper to our attention.

   Finally, we would like to thank all those with whom we had enlightening discussions during the preparation of this paper, notably Jo\"el Bella\"iche, Kevin Buzzard, Francesc Castella, Henri Darmon, Fred Diamond, Karl Rubin and Jacques Tilouine; and the two anonymous referees, whose comments improved the exposition substantially.

 \section{Asymmetric zeta elements}

  We begin by attending to some ``unfinished business'' from our earlier paper \cite{leiloefflerzerbes13}, proving some norm-compatibility relations for motivic cohomology classes extending those of \S 3 of \emph{op.cit.}.

  \subsection{Definitions}
   \label{sect:asymmzetadefs}

   Recall that in \cite[\S 2.7]{leiloefflerzerbes13} we have defined classes ${}_c \Xi_{m, N, j} \in \CH^2(Y_1(N)^2 \otimes \QQ(\mu_m), 1)$, for $m \ge 1, N \ge 5$ integers, $j \in \ZZ / m\ZZ$ and $c > 1$ coprime to $6mN$.

   In the present work, it will be convenient to extend this construction, in a rather trivial way, to give elements of higher Chow groups of products $Y_1(N) \times Y_1(N')$. We thus make the following definition:

   \begin{definition}
    \label{def:xi}
    For $m \ge 1, N, N' \ge 5$,  $j \in \ZZ / m \ZZ$, and $c > 1$ coprime to $6mNN'$, we define
    \[ {}_c \Xi(m, N, N', j) \in \CH^2(Y_1(N) \times Y_1(N') \times \Spec \QQ(\mu_m), 1)\]
    as the image of ${}_c \Xi_{m, R, j}$, for some $R$ divisible by $N$ and $N'$ and having the same prime factors as $N N'$, under pushforward via the natural degeneracy map
    \[ Y_1(R)^2 \to Y_1(N) \times Y_1(N').\]

    When $m = 1$ we omit $m$ and $j$ from the notation and write
    \[  {}_c \Xi(N, N') \coloneqq {}_c \Xi(1, N, N', 1).\]
   \end{definition}

   Note that ${}_c \Xi(m, N, N', j)$ is independent of the choice of $R$, as a consequence of Theorem 3.1.2 of \cite{leiloefflerzerbes13}.

  \subsection{Norm-compatibility}

   In addition to the norm-compatibility relations proved in \cite[\S 3]{leiloefflerzerbes13}, we shall need a few more similar statements, describing the behaviour of the ${}_c \Xi(m, N, N', j)$ for fixed $m$ and $N$ and varying $N'$, allowing both standard and ``twisted'' pushforward maps. In order to state these relations we first introduce some notation.

   \begin{notation}
    We use the following notations.
    \begin{itemize}
     \item For $d \in (\ZZ / m\ZZ)^\times$, we let $\sigma_d \in \Gal(\QQ(\mu_m) / \QQ)$ be the automorphism given by $\zeta \mapsto \zeta^d$ for each $\zeta \in \mu_m$.
     \item For each $d \in (\ZZ / N\ZZ)^\times$, we let $\langle d \rangle$ denote the diamond bracket operator on $Y_1(N)$.
     \item The operator $T_\ell'$ (for a prime $\ell \nmid N$) or $U_\ell'$ (for $\ell \mid N$) is the Hecke operator defined in \cite[\S 3.2]{leiloefflerzerbes13}, \cite[\S 2.9]{kato04}. (These are the transposes of the more familiar Hecke operators $T_\ell$, $U_\ell$.)
    \end{itemize}
    If $N, N' \ge 1$ and $T, T'$ are Hecke correspondences acting on $Y_1(N)$ and $Y_1(N')$ respectively, then the product of $T$ and $T'$ defines a correspondence on $Y_1(N) \times Y_1(N')$, which we shall write as $(T, T')$.
   \end{notation}

   \begin{theorem}
    \label{thm:asymmetricnorm}
    Let $m \ge 1, N, N' \ge 5$ be integers, $\ell$ a prime, $j \in \ZZ / m\ZZ$, and $c > 1$ an integer coprime to $6\ell mNN'$. Let $\pr_1, \pr_2$ be the two degeneracy maps $Y_1(\ell N') \to Y_1(N')$, corresponding to $z \mapsto z$ and $z \mapsto \ell z$ respectively.
    \begin{enumerate}[(a)]
     \item We have
     \begin{multline*}
      (1 \times \pr_1)_* \left( {}_c \Xi(m, N, \ell N', j) \right) =\\
      \begin{cases}
       {}_c \Xi(m, N, N', j)  & \text{if $\ell \mid m N N'$,}\\
       \left[1 - (\langle \ell^{-1}\rangle, \langle \ell^{-1} \rangle) \sigma_\ell^{-2}\right] \cdot {}_c \Xi(m, N, N', j) & \text{if $\ell \nmid m N N'$.}
      \end{cases}
     \end{multline*}
     \item
     \begin{enumerate}[(i)]
      \item if $\ell \mid N$, then
      \[ (1 \times \pr_2)_* \left( {}_c \Xi(m, N, \ell N', j) \right) = (U_\ell', 1) \cdot {}_c \Xi(m, N, N', \ell j);\]
      \item if $\ell \nmid N$ but $\ell \mid N'$,then
      \begin{multline*} (1 \times \pr_2)_* \left( {}_c \Xi(m, N, \ell N', j) \right)\\ = (T_\ell', 1) \cdot {}_c \Xi(m, N, N', \ell j) - (\langle \ell^{-1} \rangle, U_\ell') \cdot {}_c \Xi(m, N, N', \ell^2 j);\end{multline*}
      \item if $\ell \nmid mNN'$, then
      \[ (1 \times \pr_2)_* \left( {}_c \Xi(m, N, \ell N', j) \right) =
       \left[(T_\ell', 1) \sigma_\ell^{-1} - (\langle \ell^{-1} \rangle, T_\ell') \sigma_\ell^{-2} \right] \cdot {}_c \Xi(m, N, N', j).
      \]
     \end{enumerate}
    \end{enumerate}
   \end{theorem}

   \begin{remark}
    There is also a version of the above theorem with $N$ varying instead of $N'$, i.e.~describing the degeneracy $(m, N\ell, N') \mapsto (m, N, N')$. This can be deduced immediately from the above theorem using the fact that the symmetry map $Y_1(N) \times Y_1(N') \to Y_1(N') \times Y_1(N)$ interchanges ${}_c \Xi(m, N, N', j)$ and ${}_c \Xi(m, N', N, -j)$.

    In the statement of the theorem we have excluded the case where $\ell \mid m$ but $\ell \nmid NN'$; this is not because it is any more difficult, but simply because the answer is more complicated to write down -- see Remark \ref{remark:higherpowers} below.
   \end{remark}

   The proof of Theorem \ref{thm:asymmetricnorm} will be given in Appendix \ref{sect:appendix} below, as the proof requires the consideration of certain auxilliary modular curves and cohomology classes which will not be used elsewhere in the paper.

 \section{Euler systems in motivic cohomology}
  \label{sect:motivicES}

  In this section, we'll use the asymmetric zeta elements introduced above to construct a family of motivic cohomology classes attached to a modular form and a Gr\"ossencharacter of an imaginary quadratic field, indexed by ideals of the field, and satisfying a compatibility relation involving Euler factors. However, this is not quite an ``Euler system'' in the strict sense, since our elements for different $\fn$ live in motivic cohomology groups of different varieties (rather than of one variety over extensions of the base field).

  \subsection{Setup}
   \label{sect:cmsetup}

   Let $K$ be an imaginary quadratic field, and $\psi$ a Gr\"ossencharacter of $K$ of infinity-type $(-1, 0)$ and some modulus $\ff$ (not necessarily primitive, i.e. $\ff$ need not be the conductor of $\psi$), taking values in a finite extension $L / K$. We write $\chi$ for the unique Dirichlet character modulo $N_{K/\QQ}(\ff)$ such that $\psi(\, (n)\,) = n\chi(n)$ for integers $n$ coprime to $N_{K / \QQ}(\ff)$.

   \begin{theorem}[see e.g.~{\cite[Theorem 4.8.2]{miyake06}}]
    The formal $q$-expansion
    \[ \sum_{\fa} \psi(\fa) q^{N_{K / \QQ}(\fa)},\]
    where the sum is over integral ideals of $K$ coprime to $\ff$, is the $q$-expansion of a Hecke eigenform
    \[ g \in S_2(\Gamma_1(N), \chi \varepsilon_K), \]
    where $N = N_{K / \QQ}(\ff) \cdot \disc$ and $\varepsilon_K$ is the quadratic Dirichlet character attached to $K$. This eigenform is new of level $N$ if and only if $\psi$ is primitive of conductor $\ff$.
   \end{theorem}

  \subsection{Definitions: Hecke algebras}

   We now define a quotient of cohomology which describes the Galois representations attached to twists of $\psi$ by finite-order characters.

   Let $\fn$ be an integral ideal of $K$, which we assume to be divisible by $\ff$, and let $N = N_{K / \QQ}(\fn) \cdot \disc$, which is a multiple of $N_\psi = N_{K / \QQ}(\ff) \cdot \disc$. Let $H_\fn$ be the ray class group of $K$ modulo $\fn$, and for $\fl$ an ideal of $K$ coprime to $\fn$, let $[\fl]$ denote the class of $\fl$ in $H_\fn$.

   Let $\TT_{N}$ denote the subalgebra of $\End_{\ZZ} H^1(Y_1(N)(\CC), \ZZ)$ generated by the diamond operators, the $T_\ell$ for $\ell \nmid N$, and the $U_\ell$ for $\ell \mid N$. It will be convenient to use the notation $T_\ell$, for $\ell \mid N$, to denote the same operator as $U_\ell$, so we can say that $\TT_N$ is generated by the diamond operators and the $T_\ell$ for all primes $\ell$.

   \begin{proposition}
    \label{prop:defphin}
    There exists a homomorphism $\phi_{\fn}: \TT_{N} \to \fO_L[H_\fn]$ acting on the generators as follows: for $\ell$ prime,
     \[ \phi_\fn(T_\ell) = \sum_{\fl} [\fl] \psi(\fl)\]
     where the sum is over the (possibly empty) set of ideals $\fl \nmid \fn$ of norm $\ell$; and
     \[ \phi_{\fn}(\langle d \rangle) = \chi(d)\, \varepsilon_K(d)\, [(d)].\]
   \end{proposition}

   \begin{proof}
    Each of the systems of eigenvalues obtained by specializing at characters of $H_\fn$ corresponds to a nonzero eigenform in $S_2(\Gamma_1(N), \overline{L})$, so the morphism is well-defined.
   \end{proof}

   \begin{definition}
    Define
    \[ H^1(\psi, \fn, \fO_L) := \fO_L[H_\fn] \otimes_{\TT_{N}, \phi_{\fn}} H^1(Y_1(N)(\CC), \ZZ)_*,\]
    where the lower star indicates that we use the \emph{covariant} action of Hecke operators (rather than the usual contravariant action).
   \end{definition}

   We shall also need to discuss a quotient of motivic cohomology attached to $\psi$ and another eigenform (not necessarily CM), over a cyclotomic field $\QQ(\mu_m)$. To define this, let $f$ be a cuspidal modular form of weight 2 and some level $N_f$ (not necessarily a newform) which is an eigenform for all Hecke operators. Assume $L$ is sufficiently large that the Hecke eigenvalues of $f$ lie in $\fO_L$, so we have a morphism $\phi_f : \TT_{N_f} \to \fO_L$.

   \begin{definition}
    We define
    \begin{multline*}
     H^3_{\mot}(f, \psi, m, \fn, \fO_L(2)) \coloneqq \\ \fO_L[H_\fn] \mathop{\otimes}_{(\TT_{N_f} \otimes \TT_{N}, \phi_f \otimes \phi_\fn)} H^3_{\mot}(Y_1(N_f) \times Y_1(N) \times \Spec \QQ(\mu_m), \ZZ(2))_*.
    \end{multline*}
   \end{definition}

   (Again, the lower star signifies that we use the covariant rather than contravariant action of Hecke correspondences.)

  \subsection{Definitions: degeneracy maps}

   Let us now consider two moduli $\fn$ and $\fn' = \fn \fl$, with $\fl$ prime. Let $N = N_{K / \QQ}(\fn) \cdot \disc$ as before, and $N' = N \cdot N_{K / \QQ}(\fl)$. Let $\ell$ be the rational prime below $\fl$, and let
   \[ \Lambda_{\fn} = \begin{cases}
                       \fO_L[H_{\fn}] & \text{if $\fl \mid \fn$,}\\
                       \fO_L[H_{\fn}][1/\ell] & \text{if $\fl \nmid \fn$}.
                      \end{cases}
   \]

   We also consider the formal double coset space
   \[ \mathcal{R}_{N, N'} = \ZZ\left[ \Gamma_1(N) \backslash \GL_2^+(\QQ) / \Gamma_1(N')\right].\]
   Elements of $\mathcal{R}_{N, N'}$ induce correspondences $Y_1(N') \to Y_1(N)$.

   Let $\mathcal{T}_N$ denote the commutative subalgebra of $\mathcal{R}_{N, N}$ generated by the Hecke operators $T_n$ and $\langle d \rangle$; then (by definition) $\mathcal{T}_N$ surjects onto $\TT_N$, so we may regard $\phi_\fn$ as a homomorphism $\mathcal{T}_{N} \to \Lambda_{\fn}$.

   The space $\mathcal{R}_{N, N'}$ is both a left $\mathcal{T}_{N}$-module and a right $\mathcal{T}_{N'}$-module. We may regard the degeneracy maps $\pr_1$ and $\pr_2$ as elements of $\mathcal{R}_{N, N'}$, corresponding to the matrices $\begin{pmatrix} 1 & 0 \\ 0 & 1 \end{pmatrix}$ and $\begin{pmatrix} \ell & 0 \\ 0 & 1 \end{pmatrix}$; if $\fl$ is an inert prime (so $N' = \ell^2 N$) there is a third such map $\pr_3$ corresponding to $\begin{pmatrix} \ell^2 & 0 \\ 0 & 1 \end{pmatrix}$.

   \begin{definition}
    \label{def:curlyN}
    Let $\mathcal{N}_{\fn}^{\fn'}$ denote the element of
    \[ \Lambda_{\fn} \mathop{\otimes}_{\mathcal{T}_N, \phi_\fn} \mathcal{R}_{N, N'}\]
    given by the following formulae:
    \begin{itemize}
     \item If $\fl \mid \fn$, then
     \[ \mathcal{N}_{\fn}^{\fn'} = 1 \otimes \pr_1.\]
     \item If $\fl \nmid \fn$ and $\fl$ is ramified or split in $K/\QQ$, then
     \[ \mathcal{N}_{\fn}^{\fn'} = 1 \otimes \pr_1 - \frac{[\fl] \psi(\fl)}{\ell} \otimes \pr_2. \]
     \item If $\fl \nmid \fn$ and $\fl = (\ell)$ is an inert prime, then
     \[ \mathcal{N}_{\fn}^{\fn'} = 1 \otimes \pr_1 - \frac{[\fl] \psi(\fl)}{\ell^2} \otimes \pr_3.\]
    \end{itemize}
   \end{definition}

   \begin{proposition}
    \label{prop:equivariance}
    For any $A \in \mathcal{T}_{N'}$, we have
    \[  \mathcal{N}_{\fn}^{\fn'} \cdot A = \tau(\phi_{\fn'}(A)) \cdot \mathcal{N}_{\fn}^{\fn'},\]
    where $\tau$ is the natural surjection $\fO_L[H_{\fn'}] \to \fO_L[H_{\fn}]$.

    In particular, $\mathcal{N}_{\fn}^{\fn'}$ induces maps
    \[ H^1(\psi, \fn', \fO_L)[1/\ell] \to H^1(\psi, \fn, \fO_L)[1/\ell]\]
    and
    \[ H^3_{\mot}(f, \psi, \fn', \fO_L)[1/\ell] \to H^3_{\mot}(f, \psi, \fn, \fO_L)[1/\ell],\]
    and the $1/\ell$ may be omitted when $\fl \mid \fn$.
   \end{proposition}

   \begin{proof}
    When $\fl \mid \fn$ this is immediate, since we have a commutative diagram of algebras
    \begin{diagram}
     \mathcal{T}_{N'} & \rTo^{\phi_{\fn'}} &\fO_L[H_{\fn'}] \\
     \dTo^{\sigma} && \dTo^\tau\\
     \mathcal{T}_N & \rTo^{\phi_{\fn}} & \fO_L[H_\fn]
    \end{diagram}
    where the left vertical map $\sigma$ sends each generator of $\mathcal{T}_{N'}$ to the corresponding operator in $\mathcal{T}_N$; and we have $\pr_1 \cdot A = \sigma(A) \cdot \pr_1$ for all $A \in \mathcal{T}_{N'}$ (i.e.~$\pr_1$ commutes with all Hecke operators) so we are done.

    When $\fl \nmid \fn$, the same argument works if we replace $\mathcal{T}_{N'}$ with the subalgebra $\mathcal{T}_N^\circ$ generated by all the operators except $U_\ell'$. So we must only prove the equivariance property for $U_\ell'$, which follows by a case-by-case check.

    For instance, if $\fl$ is a split prime and its conjugate $\overline{\fl}$ does not divide $\fn$ either, then we have
    \begin{align*}
      \mathcal{N}_{\fn}^{\fn'} \cdot U_\ell' &= 1 \otimes (\pr_1 \cdot U_\ell') - \frac{[\fl] \psi(\fl)}{\ell} \otimes (\pr_2 \cdot U_\ell') \\
      &= 1 \otimes (T_\ell' \cdot \pr_1 - \langle \ell^{-1} \rangle \cdot \pr_2) - \frac{[\fl] \psi(\fl)}{\ell} \otimes \ell \pr_2 \\
      &= \left(\phi_{\fn}(T_\ell') - [\fl] \psi(\fl)\right) \otimes \pr_1 - \phi_\fn(\langle \ell^{-1} \rangle) \otimes \pr_2 \\
      &= [\overline{\fl}] \psi(\overline{\fl}) \otimes \pr_1 - [\fl \overline{\fl}] \psi(\fl \overline{\fl}) \circ \pr_2 \\
      &= [\overline{\fl}] \psi(\overline{\fl}) \cdot \mathcal{N}_{\fn}^{\fn'}.
    \end{align*}
    The other cases (where $\fl$ is split with $\fl \nmid \fn$ but $\overline{\fl} \mid \fn$, or when $\fl$ is inert or ramified) follow similarly.
   \end{proof}

   We extend the definition of $\mathcal{N}_{\fn}^{\fn'}$ to any pair of moduli $\fn \mid \fn'$ in the obvious way, by composing the above maps for each prime divisor $\fl$ of $\fn' / \fn$, using the multiplication maps $\mathcal{R}[N, N'] \otimes \mathcal{R}[N', N''] \to \mathcal{R}[N, N'']$. This is well-defined, since $\pr_1 \cdot \pr_2 = \pr_2 \cdot \pr_1$ as elements of $\mathcal{R}(N, N\ell^2)$, and similarly for $\pr_3$, and Proposition \ref{prop:equivariance} extends immediately to this case.

  \subsection{Definitions: classes}

   We are now in a position to construct our compatible family of motivic cohomology classes. Let $f, K, m, \fn, \psi$ be as before. Note that $N_f$ and $N = N_{K / \QQ}(\fn) \cdot \disc$ are the levels of weight 2 cusp forms, so in particular they are both $\ge 5$.

   \begin{definition}
    Let $c > 1$ be an integer coprime to $6 m N N_f$. Let ${}_c \Xi_{m, \fn}^{f, \psi}$ be the image of the element
    \[ {}_c \Xi(m, N_f, N) = {}_c \Xi(m, N_f, N, 1) \in H^3_{\mot}(Y_1(N_f) \times Y_1(N) \times \Spec \QQ(\mu_m), \ZZ(2))\]
    in the space
    \begin{multline*}
     H^3_{\mot}(f, \psi, m, \fn, \fO_L(2)) \coloneqq \\ \fO_L[H_\fn] \mathop{\otimes}_{(\TT_{N_f} \otimes \TT_{N}, \phi_f \otimes \phi_\fn)} H^3_{\mot}(Y_1(N_f) \times Y_1(N) \times \Spec \QQ(\mu_m), \ZZ(2))_*.
    \end{multline*}
   \end{definition}

  \subsection{Norm-compatibility}

   \begin{theorem}
    \label{thm:normcompat}
    The elements ${}_c \Xi_{m, \fn}^{f, \psi}$ enjoy the following compatibility property. Let $\fn \mid \fn'$ be two ideals of $K$ divisible by $\ff$, and let $A$ be the set of primes dividing $\fn'$ but not $\fn$. Suppose that no prime in $A$ divides $m$. Then
    \[ \mathcal{N}_{\fn}^{\fn'} \left({}_c \Xi_{m, \fn'}^{f, \psi}\right) =
     \left( \prod_{\fl \in A} P_\fl\left([\fl] \sigma_\fl^{-1} N(\fl)^{-1}\right)\right) {}_c \Xi_{m, \fn}^{f, \psi}
    \]
    as elements of
    \[ H^3_{\mot}(f, \psi, m, \fn, \fO_L(2)) \otimes_{\fO_L} \fO_L\left[ \tfrac 1{N(\fl)}: \fl  \in A\right],\]
    where $P_\fl$ denotes the Euler factor of $f \otimes \psi$ at $\fl$, and $\sigma_\fl \in \Gal(\QQ(\mu_m) / \QQ)$ is the element $\zeta \mapsto \zeta^{N_{K/\QQ}(\fl)}$.
   \end{theorem}

   \begin{proof}
    It suffices to consider the case where $\fn' = \fl \fn$ for $\fl$ a prime. As usual, write $N = N_{K/\QQ}(\fn) \cdot \disc$, and similarly $N' = N_{K/\QQ}(\fn') \cdot \disc$.

    If $\fl \mid \fn$, then $\mathcal{N}_{\fn}^{\fn'}$ is the map induced by
    \[ 1 \times \pr_1: Y_1(N_f) \times Y_1(N') \to Y_1(N_f) \times Y_1(N)\]
    and $N_f N$ and $N_f N'$ have the same prime factors, so we are done, by the first case of part (a) of Theorem \ref{thm:asymmetricnorm}.

    Hence we may assume that $\fl \nmid \fn$. In this case we have $\ell \nmid m$, where $\ell$ is the rational prime below $\fl$; and $\sigma_{\fl}$ is the usual arithmetic Frobenius $\sigma_\ell$ at $\ell$ if $\fl$ is split or ramified, and $\sigma_\fl = \sigma_\ell^2$ if $\fl$ is inert.

    We now have eight cases to consider (since $\fl$ may be ramified in $K$, inert, split with $\overline\fl \nmid \fn$, or split with $\fl \mid \fn$, and $\ell$ may or may not divide $N_f$). Each of these can be handled using different cases of Theorem \ref{thm:asymmetricnorm}. We describe the argument in the case where $\fl$ is split, $\overline\fl \nmid \fn$, and $\ell \nmid N_f$:
    \begin{multline*}
     \mathcal{N}_{\fn}^{\fl \fn} \left({}_c \Xi_{m, \fl \fn}^{f, \psi}\right) \\
     \begin{aligned}
      &= \left[1 \otimes (1 \times \pr_1)_* - \frac{\psi(\fl) [\fl]}{\ell} \otimes (1 \times \pr_2)_* \right] {}_c \Xi(m, N_f, \ell N)\\
      & \begin{aligned}=\Bigg[ 1 \otimes (1 - &(\langle \ell \rangle \times \langle \ell \rangle)_* \sigma_\ell^{-2}) \\
&- \frac{\psi(\fl) [\fl]}{\ell} \otimes ( (T_\ell \times 1)_* \sigma_\ell^{-1} - (\langle \ell \rangle \times T_\ell)_* \sigma_\ell^{-2})\Bigg] {}_c \Xi(m, N_f, N)\end{aligned}\\
      & \begin{aligned}=\Bigg[ 1 - &\varepsilon_\ell(f) \frac{[\fl \overline\fl]\psi(\fl \overline\fl)}{\ell} \sigma_\ell^{-2} \\ &-\frac{\psi(\fl)[\fl]}{\ell}\Big(a_\ell(f) \sigma_\ell^{-1} - \varepsilon_\ell(f)\sigma_\ell^{-2} (\psi(\fl) [\fl] + \psi(\overline\fl)[\overline\fl])\Big)\Bigg] (1 \otimes {}_c \Xi(m, N_f, N))\end{aligned}\\
      &= \left[ 1 - a_\ell(f) \sigma_\ell^{-1} \frac{\psi(\fl)[\fl]}{\ell} + \ell \varepsilon_\ell(f) \sigma_\ell^{-2} \left(\frac{[\fl] \psi(\fl)}{\ell}\right)^2\right] {}_c \Xi_{m, \fn}^{f, \psi}.
     \end{aligned}
    \end{multline*}
    The other cases, which are very similar, we leave to the reader.
   \end{proof}

   \begin{remark}
    In the remainder of this paper, we shall in fact only use the elements ${}_c \Xi^{f, \psi}_{m, \fn}$ for $m = 1$. We have worked with general $m$ above since we intend to use the classes for $m = p^k$ in a future work to study the Iwasawa theory of $f$ over the $\Zp^2$-extension of $K$.
   \end{remark}


 \section{Hecke algebras and Ihara's lemma}
  \label{sect:hecke}

  We now collect some results on the Hecke action on the integral cohomology groups of modular curves. Modulo minor modifications all of the results below can be found in \cite[Chapter 2]{wiles95}.

  We adopt the shorthand notation $H^1(Y_1(N))$ for $H^1(Y_1(N)(\CC), \ZZ)$.

  \subsection{Freeness results}

   Let $N \ge 5$ be an integer. Note that $H^1(Y_1(N))$ is a free $\ZZ$-module, since for $N \ge 5$ the group $\Gamma_1(N)$ has no torsion.

   As above, let
   \[ \TT_N \subseteq \End_{\ZZ} H^1(Y_1(N))\]
   be the commutative $\ZZ$-subalgebra generated by the operators $\langle d \rangle$ for $d \in (\ZZ / N\ZZ)^\times$, $T_\ell$ for primes $\ell \nmid N$, and $U_\ell$ for primes $\ell \mid N$.

   \begin{remark}
    Note that there are ``covariant'' and ``contravariant'' actions of Hecke operators on $H^1(Y_1(N))$; but the two actions are interchanged by the Atkin--Lehner involution, so the subalgebras of $\End_{\ZZ} H^1(Y_1(N))$ generated by the two actions of Hecke operators are isomorphic. We shall generally regard $H^1(Y_1(N))$ as a $\TT_N$-module via the contravariant action of Hecke operators; if we mean to regard it as a $\TT_N$-module via the covariant action, we shall write it as $H^1(Y_1(N))_*$ (lower star for pushforward).
   \end{remark}

   \begin{definition}
    \label{def:nonEis}
    A maximal ideal $\cI$ of $\TT_N$ of residue characteristic $p > 2$ is said to be \emph{non-Eisenstein} if there exists a continuous and absolutely irreducible representation
    \[ \overline{\rho}_{\cI}: G_{\QQ} \to \GL_2(\TT_N / \cI) \]
    such that for $\ell \nmid Np$ we have
    \[ \Tr \overline{\rho}_{\cI}(\sigma_{\ell}^{-1}) = T_\ell \bmod \cI\]
    and
    \[ \det \overline{\rho}_{\cI}(\sigma_{\ell}^{-1}) = \ell \langle \ell \rangle \bmod \cI.\]
   \end{definition}

   Given such an ideal, we write $(\TT_N)_{\cI}$ for the $\cI$-adic completion of the localization of $\TT_N$ at $\cI$, which is a finite-rank free $\Zp$-algebra. Similarly, we write $H^1(Y_1(N))_{\cI}$ for the completion of the homology group at $\cI$. As $p \in \cI$, this is a free $\Zp$-module, and is isomorphic to the corresponding \'etale cohomology group $H^1_\et(\overline{Y_1(N)}, \Zp)_{\cI}$; in particular it has a $(\TT_N)_\cI$-linear action of $\Gal(\overline{\QQ} / \QQ)$.

   \begin{proposition}
    Let $\cI$ be a non-Eisenstein maximal ideal of $\TT_N$. Then the maps
    \[ H^1_c(Y_1(N))_{\cI} \to H^1(X_1(N))_{\cI} \to H^1(Y_1(N))_{\cI}\]
    are isomorphisms.
   \end{proposition}

   \begin{proof}
    This is essentially the Manin--Drinfeld theorem: we can always find a supply of primes $\ell$ such that $1 + \ell - T_\ell$ annihilates the boundary cohomology group $H^1(\partial X_1(N))$, but non-Eisensteinness guarantees that we can find some such $\ell$ with $1 + \ell - T_\ell$ not in $\cI$, so it is invertible after localizing at $\cI$.
   \end{proof}

   We now invoke the following deep theorem of Wiles and others, originating in Mazur's work on the Eisenstein ideal:

   \begin{theorem}
    \label{thm:gorenstein}
    If $\cI$ is a non-Eisenstein maximal ideal and $p \nmid N$, then $(\TT_N)_{\cI}$ is a Gorenstein ring, and $H^1(Y_1(N))_{\cI}$ is a free $(\TT_N)_{\cI}$-module of rank 2. The same also holds if we replace $Y_1(N)$ with $Y(\Gamma)$ for any subgroup intermediate between $\Gamma_1(N)$ and $\Gamma_0(N)$.
   \end{theorem}

   \begin{proof}
    See e.g.~\cite[Theorem 2.1]{wiles95}. (The result is stated there in terms of the Hecke module $\Hom(J_1(N)[p^\infty], \Qp/\Zp)_{\cI}$, which is isomorphic to $H^1(X_1(N))_{\cI}$, but the preceding proposition shows that we may replace $X_1(N)$ with $Y_1(N)$.)
   \end{proof}

  \subsection{Degeneracy maps}

   We now compare Hecke algebras and Hecke modules at different levels. Throughout this section, $N$ will be an integer $\ge 5$ and $\ell$ will be a prime not dividing $N$. We write $Y_1(N; \ell)$ for the modular curve of level $\Gamma_1(N) \cap \Gamma_0(\ell)$.

   We begin by recalling some standard results:

   \begin{lemma}[Ihara]
    \label{lemma:ihara}
    The map
    \[ (\pr_1)_* \oplus (\pr_2)_* : H_1(X_1(N; \ell)) \to H_1(X_1(N))^{\oplus 2}\]
    is a surjection.\qed
   \end{lemma}

   \begin{lemma}[Wiles]
    \label{lemma:wiles}
    For any odd prime $p \ne \ell$, and any $r \ge 1$, there is an exact sequence
    \[ H_1(Y_1(N\ell^r; \ell^{r+1}), \Zp) \rTo H_1(Y_1(N\ell^r), \Zp)^{\oplus 2} \rTo H_1(Y_1(N\ell^{r-1}), \Zp),\]
    where the maps are respectively $x \mapsto ((\pr_1)_* x, (\pr_2)_*x)$ and $(u, v) \mapsto (\pr_2)_*(u) - (\pr_1)_*(v)$.\qed
   \end{lemma}

   (We have stated these lemmas slightly differently from Wiles, who formulates Ihara's lemma in terms of morphisms of Jacobians, and Lemma \ref{lemma:wiles} in terms of group cohomology with $\Qp/\Zp$ coefficients; for the formulations above see \cite[Lemma 4.28]{darmondiamondtaylor97}.)

   \begin{corollary}
    \label{cor:ulseq}
    The following sequence is exact for any odd prime $p \ne \ell$ and any $r \ge 1$:
    \[ H_1(Y_1(N\ell^r; \ell^{r+1}), \Zp) \rTo^{(\pr_1)_* - \frac{U_\ell}{\ell} (\pr_2)_*} H_1(Y_1(N\ell^r), \Zp)
     \rTo^{(\pr_2)_*} H_1(Y_1(N\ell^{r-1}), \Zp).\]
   \end{corollary}

   \begin{proof}
    By applying the matrix $\begin{pmatrix} 1 & -\tfrac{U_\ell}{\ell} \\ 0 & 1 \end{pmatrix}$ to the middle term of the exact sequence of Lemma \ref{lemma:wiles} we deduce the exact sequence
    \begin{multline*}
     H_1(Y_1(N\ell^r; \ell^{r+1}), \Zp) \rTo^{\left((\pr_1)_* - \frac{U_\ell}{\ell} (\pr_2)_*, (\pr_2)_*\right)}
     H_1(Y_1(N\ell^r), \Zp)^{\oplus 2} \\
     \rTo^{\left(\begin{smallmatrix} (\pr_2)_* \\ 0 \end{smallmatrix}\right)} H_1(Y_1(N\ell^{r-1}), \Zp),
    \end{multline*}
    which implies the exactness of the desired sequence.
   \end{proof}

   \begin{lemma}
    \label{lemma:diamonds}
    The pushforward map
    \[ H^1(Y_1(N\ell^{r+1})) \to H^1(Y_1(N\ell^r; \ell^{r+1}))\]
    is surjective for any $r \ge 0$.
   \end{lemma}

   \begin{proof}
    We prove the dual version of the statement: the cokernel of the pullback map
    \[ H^1_c(Y_1(N\ell^r; \ell^{r+1})) \to H^1_c(Y_1(N\ell^{r+1}))\]
    is torsionfree. This follows from the ``modular symbol'' isomorphism
    \[ H^1_c(Y(\Gamma)) = \Hom_{\Gamma}(\operatorname{Div}^0(\mathbf{P}^1_{\QQ}), \ZZ),\]
    valid for any torsion-free congruence subgroup $\Gamma$, which implies that we have an isomorphism $H^1_c(Y_1(N\ell^{r};\ell^{r+1})) = H^1_c(Y_1(N\ell^{r+1}))^\Delta$, where $\Delta$ is the kernel of $(\ZZ / \ell^{r+1} \ZZ)^\times \to (\ZZ / \ell^r \ZZ)^\times$.
   \end{proof}

   \begin{remark}
    Compare Lemma 4.30(b) of \cite{darmondiamondtaylor97}, which shows that the cokernel of the map $H_1(X_1(N), \Zp) \to H_1(X_H(N), \Zp)$ is Eisenstein for $H$ any subgroup of $(\ZZ / N\ZZ)^\times$.
   \end{remark}

   \begin{lemma}[{Ribet, cf.~\cite[Lemma, p492]{wiles95}}]
    \label{lemma:omitprimes}
    Let $\Sigma$ be any finite set of primes not dividing $N$, and let $\TT_{X_1(N)}$ be the quotient of $\TT_N$ that acts faithfully on $H^1(X_1(N))$. Then the subalgebra of $\TT_{X_1(N)}$ generated by the diamond operators and the $T_q$ for $q \notin \Sigma$ has finite index in $\TT_{X_1(N)}$, and this index is 1 if $2 \notin \Sigma$ and a power of 2 otherwise.\qed
   \end{lemma}

   In order to apply all of these results at once, we will need to localize at a non-Eisenstein maximal ideal, after which there is no difference between $H^1$ and $H_1$, or between $Y_1(N)$ and $X_1(N)$.
   We now define some Hecke algebras that we shall need.

   \begin{definition}
    For $r \ge 1$, let $\TT_{N\ell^r}^\circ$ be the subalgebra of $\TT_{N\ell^r}$ generated by the diamond operators and the $T_q$ for $q \ne \ell$ (including the operators $T_q = U_q$ for $q \mid N$), but not $U_\ell$..

    We write $\widetilde \TT_N$ for the ring $\TT_N[X] / (X^2 - T_\ell X + \ell \langle \ell \rangle)$.
   \end{definition}

   There is a commutative diagram
   \[\begin{diagram}
    \TT_{N\ell}^\circ &\rTo& \TT_N \\
    \dInto & & \dInto \\
    \TT_{N\ell} & \rOnto^{\lambda_2} &\widetilde \TT_N,
   \end{diagram}\]
   where the top horizontal arrow is the natural map, and the map $\lambda$ is defined by $\lambda(U_\ell) = X$.

   Let $\cI$ be a non-Eisenstein maximal ideal of $\TT_N$ of residue characteristic $p \nmid N\ell$. We can regard $\cI$ also as a maximal ideal of $\TT_N^\circ$. By Lemma \ref{lemma:omitprimes}, the morphism of completions $(\TT_{N_\ell}^\circ)_{\cI} \onto (\TT_N)_{\circ}$ is a surjection.

   We can now proceed to the first main result of this section, which asserts the surjectivity of an ``$\ell$-stabilization'' map.

   \begin{theorem}
    \label{thm:normmaps1}
    The map
    \[ \beta: (\widetilde \TT_N)_{\cI} \otimes_{\TT_{N\ell}} H^1(Y_1(N\ell))_* \to (\widetilde \TT_N)_{\cI} \otimes_{\TT_{N}} H^1(Y_1(N))_*\]
    defined by
    \[ (\pr_1)_* - \frac{T_\ell - X}{\ell} (\pr_2)_*\]
    is an isomorphism.
   \end{theorem}

   \begin{proof}
    Firstly, we note that $\beta$ is well-defined, since the map $\gamma: H^1(Y_1(N\ell))_* \to (\widetilde \TT_N)_{\cI} \otimes_{\TT_{N}} H^1(Y_1(N))_*$ defined by $(\pr_1)_* - \frac{T_\ell - X}{\ell} (\pr_2)_*$ satisfies $\gamma \circ U_\ell = X \gamma$ (cf.~Proposition \ref{prop:equivariance} above). Moreover, $\beta$ is an isomorphism after inverting $p$; and its source and target are both free $(\widetilde \TT_N)_{\cI}$-modules by Theorem \ref{thm:gorenstein}, and in particular free $\Zp$-modules, so $\beta$ is injective.

    It remains to check that $\beta$ is surjective. This is essentially a lightly disguised form of Ihara's lemma. We do this by constructing a module for the (somewhat artificial) algebra $\widetilde \TT_N$ (following the argument used by Wiles to prove an analogous statement for $\ell = p$, cf~\cite[p490]{wiles95}): we let $\widetilde \TT_N$ act on the module $H^1(Y_1(N))_*^{\oplus 2}$ with $\TT_N$ acting via the covariant action and $X$ acting by the matrix $\begin{pmatrix} T_\ell & -\langle \ell \rangle \\ \ell & 0 \end{pmatrix}$. The map
    \[ (\pr_1)_* \oplus (\pr_2)_*:  H^1(Y_1(N \ell))_* \to H^1(Y_1(N))_*^{\oplus 2}\]
    is then a morphism of $\TT_{N \ell}$-modules, and Ihara's lemma (combined with Lemma \ref{lemma:diamonds}) shows that after localizing at $\cI$ it is surjective. However, $H^1(Y_1(N))^{\oplus 2}_*$ is isomorphic to $\widetilde \TT_N \otimes_{\TT_N} H^1(Y_1(N))_*$, and, unravelling the definitions, we find that the composite map is exactly $\beta$.
   \end{proof}

   Our second result of this section concerns ``$\ell$-depletion'' of eigenforms of level divisible by $\ell$. We first introduce a little more notation. Let $r \ge 1$. There is a map
   \[ \phi_r: \TT_{N\ell^{r+1}} \to \TT_{N\ell^r}, \]
   which maps the $\langle d \rangle$ operators and the $T_q$ for $q \ne \ell$ to themselves, and which maps $U_\ell$ to 0.

   \begin{theorem}
    \label{thm:normmaps2}
    For any $r \ge 1$, and any non-Eisenstein maximal ideal $\cI$ of $\TT_{N\ell^r}$, the map
    \[ \beta_r: (\TT_{N\ell^r})_{\cI} \mathop{\otimes}_{\displaystyle (\TT_{N\ell^{r+1}}, \phi)} H^1(Y_1(N\ell^{r+1}))_* \to (\TT_{N\ell^r})_{\cI} \mathop{\otimes}_{\displaystyle \TT_{N\ell^{r}}} H^1(Y_1(N\ell^{r}))_*\]
    is a bijection.
   \end{theorem}

   \begin{proof}
    As in the previous theorem, we first note that the map $\beta_r$ is well-defined (by the same calculation as in Proposition \ref{prop:equivariance}), its source and target are free $\Zp$-modules of finite rank, and it is a bijection after inverting $p$. Thus $\beta_r$ is injective.

    We now prove the surjectivity of $\beta_r$. We know that
    \[ \beta_r\left( H^1(Y_1(N \ell^{r+1}))_{\cI}\right) = H^1(Y_1(N\ell^r))_{\cI}^{(\pr_2)_* = 0}\]
    by Corollary \ref{cor:ulseq}. So it suffices to show that the submodule $H^1(Y_1(N \ell^r))_{\cI}^{(\pr_2)_* = 0}$ spans $H^1(Y_1(N \ell^r))_{\cI}$ as a $(\TT_{N \ell^r})_{\cI}$-module, or equivalently as a $\Zp[U_\ell]$-module.

    We prove this by induction on $r$. Let $x \in H^1(Y_1(N \ell^r))_{\cI}$ be arbitrary. We want to write
    \[ x = a_0 + U_\ell a_1 + \dots + U_\ell^r a_r\]
    for some $a_1, \dots, a_r \in H^1(Y_1(N \ell^r))_{\cI}^{(\pr_2)_* = 0}$. Equivalently, we want to find elements $a_1, \dots, a_r \in H^1(Y_1(N \ell^r))_{\cI}^{(\pr_2)_* = 0}$ such that
    \[ (\pr_2)_*\left( x -\left(U_\ell a_1 + \dots + U_\ell^r a_r\right)\right) = 0.\]

    However, we have
    \begin{multline*}
     (\pr_2)_*\left( x - \left(U_\ell a_1 + \dots + U_\ell^r a_r\right)\right) \\
     = (\pr_2)_*(x) - \ell \left[ (\pr_1)_*(a_{1}) + \dots + U\ell^{r-1}  (\pr_1)_*(a_{r})\right].
    \end{multline*}
    By the induction hypothesis, there exist $b_0, \dots, b_{r-1} \in H^1(Y_1(N\ell^{r-1}))_{\cI}^{(\pr_2)_* = 0}$ such that $(\pr_2)_*(x) = b_0 + \dots + U_\ell^{r-1} b_{r-1}$. (This statement is trivially true for $r = 1$, if we understand $\pr_2$ as the zero map.) So if we can choose the $a_i$ such that $(\pr_1)_*(a_i) = \ell^{-1} b_{i-1}$, we are done.

    So it suffices to show that
    \[ (\pr_1)_* : H^1(Y_1(N\ell^r))_{\cI}^{(\pr_2)_* = 0} \to H^1(Y_1(N\ell^{r-1}))_{\cI}^{(\pr_2)_* = 0}\]
    is surjective for all $r \ge 1$ (where, again, we understand the right-hand side as the whole of $H^1(Y_1(N\ell^{r-1}))_{\cI}$ if $r = 1$). This is immediate from Ihara's lemma if $r = 1$; for $r \ge 2$ it follows from Lemma \ref{lemma:wiles}.
   \end{proof}

   \begin{corollary}
    \label{cor:normmaps3}
    For any non-Eisenstein maximal ideal $\cI$ of $\TT_N$, the map
    \[ (\TT_N)_{\cI} \mathop\otimes_{\displaystyle (\TT_{N\ell^2}, \phi)} H^1(Y_1(N\ell^2))_* \to (\TT_N)_{\cI} \mathop\otimes_{\TT_{N}} H^1(Y_1(N))\]
    given by
    \[ (\pr_1)_* - \tfrac{T_\ell}{\ell} (\pr_2)_* + \tfrac{\langle \ell \rangle}{\ell} (\pr_3)_* \]
    is a bijection.
   \end{corollary}

   \begin{proof}
    This follows directly from Theorem \ref{thm:normmaps1} and case $r = 1$ of Theorem \ref{thm:normmaps2}: combining these theorems gives the bijectivity of the above map after tensoring with $\widetilde \TT_N$, but $\widetilde \TT_N$ is free of rank 2 over $\TT_N$ and hence faithfully flat. Alternatively, a direct argument using Ihara's lemma and lemma \ref{lemma:wiles} is given in \cite[(2.14)]{wiles95} (see also \cite[\S 4.5]{darmondiamondtaylor97}).
   \end{proof}

  \subsection{Hida theory}

   We now prove an analogue of the above results in the case setting of Hida theory, where we consider a limit over all $p$-power levels. Here $p$ will be an odd prime not dividing $N$.

   \begin{definition}
    Let
    \[ H^1_{\ord}(Y_1(Np^\infty)) = e_{\ord} \cdot \varprojlim_{r \ge 1} H^1(Y_1(Np^r), \Zp)_*,\]
    where $e_{\ord} \coloneqq \lim_{n \to \infty} (U_p)^{n!}$ is Hida's ordinary projector.
   \end{definition}

   \begin{remark}
    Note that we are using the covariant action of the Hecke algebra here, and the covariant action of $U_p$ coincides with the contravariant action of $U_p'$, so this is the same module as the one denoted $e'_{\ord} \cdot GES_p(N, \Zp)$ in \cite{ohta00} and in our previous paper.
   \end{remark}

   We let $\TT_{Np^\infty}$ be the subalgebra of $\End_{\Zp} H^1_{\ord}(Y_1(Np^\infty))$ generated by the $\langle d \rangle$ and $T_n$ operators.

   \begin{definition}
    \label{def:pdist}
    Let $\cI$ be a characteristic $p$ maximal ideal of the Hecke algebra $\TT_{Np^r}$, for $r \ge 1$. We say $\cI$ is \emph{$p$-ordinary} if $U_p \notin \cI$. We say $\cI$ is \emph{$p$-distinguished} if it is ordinary and non-Eisenstein, and the restriction of the Galois representation $\overline{\rho}_{\cI}$ to a decomposition group $D_p$ at $p$ satisfies
    \[ \overline{\rho}_{\cI}|_{D_p} \cong \begin{pmatrix} \chi_1 & * \\ 0 & \chi_2\end{pmatrix},\]
    with $\chi_1$ and $\chi_2$ distinct characters of $D_p$.
   \end{definition}

   The following theorem summarizes some of the major results of Hida theory:

   \begin{theorem}
    \label{thm:hidatheory}
    The module $H^1_{\ord}(Y_1(Np^\infty))$ is a finite-rank free module over the Iwasawa algebra $\Lambda = \Zp[[(1 + p\Zp)^\times]]$ (with the module structure given by the diamond operators). The algebra $\TT_{Np^\infty}$ is a finite flat $\Lambda$-algebra, and its maximal ideals biject with the $p$-ordinary maximal ideals of $\TT_{Np}$.

    If $\cI$ is a $p$-distinguished maximal ideal, then $(\TT_{Np^\infty})_{\cI}$ is Gorenstein, and the $(\TT_{Np^\infty})_{\cI}$-module $H^1_{\ord}(Y_1(Np^\infty))_{\cI}$ is free of rank 2.
   \end{theorem}

   \begin{proof}
    For the first part of the theorem, we refer to \S 1 of \cite{ohta00}. The finiteness and freeness of  $H^1_{\ord}(Y_1(Np^\infty))$ over $\Lambda$ is Theorem 1.3.5 of \emph{op.cit.}; the fact that $\TT_{Np^\infty}$ (denoted by $e^*\mathcal{H}^*(N; \Zp)$ in \emph{op.cit.}) is finite and flat over $\Lambda$ is Theorem 1.5.7. Moreover, since $\TT_{Np^\infty}$ is a finite flat algebra over a complete local ring, its maximal ideals biject with the maximal ideals of the Artinian ring $\TT_{Np^\infty} / J$ where $J = (p, X)$ is the maximal ideal of $\Lambda$; Theorem 1.5.7(iii) of \emph{op.cit.} shows that $\TT_{Np^\infty} / J = e_{\ord} \cdot \TT_{Np} / p$, whose maximal ideals are precisely the $p$-ordinary maximal ideals of $\TT_{Np}$.

    For the statement on freeness, we refer to \cite[Proposition 3.1.1]{emertonpollackweston06}, where the result is deduced from \cite[Theorem 2.1]{wiles95}.
   \end{proof}

   \begin{proposition}
    \label{prop:normmapshida}
    If $\cI$ is $p$-distinguished, then Theorems \ref{thm:normmaps1} and \ref{thm:normmaps2} hold with $Np^r$ in place of $N$, for any $r \ge 1$.
   \end{proposition}

   \begin{proof}
    The only ingredient of the proofs of the two theorems which required the assumption $p \nmid N$ was the freeness result of Theorem \ref{thm:gorenstein}. However, if $\cI$ is $p$-distinguished, then we know that $H^1(Y_1(Np^\infty))_{\cI}$ is free over $(\TT_{Np^\infty})_{\cI}$ by Theorem \ref{thm:hidatheory}, and the control theorem (Theorem 1.5.7(iii) of \cite{ohta00}) then implies that $H^1(Y_1(Np^r))_{\cI}$ is free over $(\TT_{Np^r})_{\cI}$.
   \end{proof}

   We also have a companion result relating forms of level prime to $p$ and level divisible by $p$.

   \begin{proposition}
    \label{prop:normmapshidap}
    Let $\cI$ be a non-Eisenstein maximal ideal of $\widetilde \TT_N$ of residue characteristic $p \nmid N$, such that $X \notin \cI$. Then the ideal of $\TT_{Np}$ corresponding to $\cI$ is ordinary and $p$-distinguished; we have $T_p - X \in p \cdot (\widetilde \TT_N)_{\cI}$; and the morphism $\beta$ of Theorem \ref{thm:normmaps1} gives an isomorphism
    \[ (\widetilde \TT_N)_{\cI} \otimes_{\TT_{Np}} H^1(Y_1(Np))_* \to (\widetilde \TT_N)_{\cI} \otimes_{\TT_{N}} H^1(Y_1(N))_*.\]
   \end{proposition}

   \begin{proof}
    This is clear by the same argument as Theorem \ref{thm:normmaps1}. (The only subtle point is that $\cI$ is $p$-distinguished as an ideal of $\TT_{Np}$; but it is ordinary since $X \notin \cI$, and of the two characters appearing in the semisimplification of $\widetilde \rho_{\cI} |_{D_p}$, one is unramified at $p$ and the other is the product of an unramified character and inverse of the mod $p$ cyclotomic character, so they are certainly distinct.)
   \end{proof}

 \section{Euler systems in \'etale cohomology}

  We now use the Hecke algebra theory of the previous section to show that if we apply the $p$-adic \'etale regulator map to the Euler system of \S \ref{sect:motivicES} and localize at a suitably chosen prime ideal of the Hecke algebra, the resulting family of classes -- all living on different modular curves -- can be ``massaged'' into an Euler system in the more conventional sense, a family of classes in the cohomology of one fixed Galois representation over varying extensions of the field $K$.

  \subsection{CM ideals of Hecke algebras}

   Let $K, L, \psi, \ff$ be as in \S \ref{sect:cmsetup} above. We fix primes $\fP \mid \fp \mid p$ of $L$, $K$ and $\QQ$ respectively, with $p \ge 5$, $p$ unramified in $K$, and $(\ff, p) = 1$.

   For convenience we shall write $E$ for the field $L_{\fP}$, $\cO = \fO_{L, \fP}$ for its ring of integers, and $\mathbf{k} = \fO_L / \fP$ for its residue field.

   Let us write $\psi_{\fP}$ for the continuous $E$-valued character of $K^\times \backslash \AA_{K, \mathrm{fin}}^\times$ defined by
   \[ \psi_{\fP}(x) = x_{\fp}^{-1} \psi(x).\]

   \begin{definition}
    \label{def:In}
    Let $\fn$ be any ideal of $K$ divisible by $\ff$, and let $N = N_{K/\QQ}(\fn) \cdot \disc$ as before. Let $\cI_{\fn}$ be the maximal ideal of the Hecke algebra $\TT_{N}$ given by the kernel of the composite map
    \[ \TT_{N} \rTo^{\phi_{\fn}}  \fO_L[H_{\fn}] \rTo \fO_L \rTo \fO_L / \fP,\]
    where $\phi_{\fn}$ is as defined in \S \ref{sect:cmsetup} and the map $\fO_L[H_{\fn}] \to \fO_L$ is the augmentation map.
   \end{definition}

   \begin{proposition}
    For any $\fn$ as above, the ideal $\cI_{\fn}$ is a non-Eisenstein maximal ideal in the sense of Definition \ref{def:nonEis}. If $p$ is split and $\fp \mid \fn$, but $\fpb \nmid \fn$, then $\cI_\fn$ is ordinary and $p$-distinguished.
   \end{proposition}

   \begin{proof}
    We can interpret $\psi_{\fP}$ as a character of $\Gal(\overline{K} / K)$ via class field theory.\footnote{We normalize the global Artin map $\AA_K^\times / K^\times \to \Gal(\overline{K} / K)^{\ab}$ in the geometric fashion, so uniformisers map to geometric Frobenius elements.} Then $\Ind_{K}^\QQ(\psi_\fP \bmod \fP)$ is the unique semisimple Galois representation with values in $\TT_{N} / \cI_\fn \cong \cO_{L} / \fP$ satisfying the trace and determinant condition of the $\overline{\rho}_{\cI}$ of Definition \ref{def:nonEis}. By Mackey theory, this induced representation is irreducible if and only if $\psi_\fP$ and its conjugate are distinct modulo $\fP$.

    If $p$ is split, then $\psi_\fP$ is ramified at $\fp$ (its restriction to inertia at $\fp$ is the inverse cyclotomic character) but $\psi_{\fP} \circ \sigma$ is not; hence these two characters are not even congruent locally at $p$. Thus $\cI_\fn$ is non-Eisenstein; and if $\fp \mid \fn$ (so that $p \mid N$ and $\phi_\fn(U_p) = \psi(\fpb) \bmod \cI_\fn$) then it is ordinary and $p$-distinguished.

    If $p$ is inert, then the restriction of $\psi_{\fP}$ to $\fO_{K, p}^\times$ is the direct sum $\omega_2^{-1} \oplus \omega_2^{-p}$, where $\omega_2$ is the Teichm\"uller character of $(\fO_{K} / p)^\times \cong \mathbf{F}_{p^2}^\times$. The characters $\omega_2$ and $\omega_2^p$ are distinct, and they are interchanged by the conjugation action of the Frobenius element of $D_p / I_p$. Hence $\Ind_{K}^\QQ(\psi_\fP \bmod \fP)$ is irreducible as a representation of $D_p$, and in particular it is irreducible as a representation of $\Gal(\overline{K} / K)$.
   \end{proof}

   \begin{remark}
    If $p$ is split, then Proposition 5.1.2 also holds if $\fp \mid \ff$, as long as we assume that $\fpb \nmid \ff$ and $\psi |_{\cO_{K, \fp}^\times}$ is not congruent to the Teichm\"uller character modulo $\fP$.
   \end{remark}

  \subsection{Patching CM Hecke modules}
   \label{sect:patching}

   We now apply the integral Hecke theory results of Section \ref{sect:hecke} to show that we can patch together the modules $H^1(\psi, \fn, \fO_L)$ after localizing at $\cI_\fn$, and identify them (non-canonically) with Galois modules induced from abelian extensions of $K$. We continue to assume that $\fn$ is an integral ideal of $K$ divisible by $\ff$.

   \begin{definition}
    Let $H_\fn^{(p)}$ denote the largest quotient of $H_\fn$ whose order is a power of $p$, and let $\Lambda_{\fn}^{\fP} = \cO[H_\fn^{(p)}]$.
   \end{definition}

   The ring $\Lambda_{\fn}^{\fP}$ is a finite, flat, local $\cO$-algebra. We let
   \[ \phi_{\fn}^{\fP} : \TT_{N} \otimes \Zp \to \Lambda_{\fn}^{\fP}\]
   be the composition of the map $\phi_{\fn}$ defined above with the natural map $\fO_L[H_{\fn}] \to \Lambda_{\fn}^{\fP}$.

   \begin{definition}
    For each $\fn$ as above, define
    \[
     H^1(\psi, \fn, \fP) \coloneqq
     \Lambda_{\fn}^{\fP}  \otimes_{(\TT_{N} \otimes \Zp, \phi_{\fn}^{\fP})} H^1_{\et}(\overline{Y_1(N)}, \Zp(1))_*,
    \]
    where the lower star signifies that we consider $H^1_{\et}(\overline{Y_1(N)}, \Zp(1))$ as a $\TT_{N}$-module via the covariant action.
   \end{definition}

   \begin{proposition}
    \label{prop:cmfree}
    Suppose either that $p$ is inert and $(p, \fn) = 1$, or $p$ is split and $(\fpb, \fn) = 1$. Then the module $H^1(\psi, \fn, \fP)$ is free of rank 2 over $\Lambda_{\fn}^{\fP}$.
   \end{proposition}

   \begin{proof}
    Since $\Lambda_{\fn}^{\fP}$ is a complete local ring, and the preimage of its maximal ideal under $\phi_{\fn}^{\fP}$ is the ideal $\cI_{\fn}$, the map $H^1_{\et}(\overline{Y_1(N)}, \Zp(1)) \to H^1(\psi, \fn, \fP)$ factors through the completion at $\cI_\fn$. However, since $\fP$ is assumed to be non-Eisenstein, the completion of $H^1_{\et}(\overline{Y_1(N)}, \Zp(1))$ at $\cI_{\fn}$ is free of rank 2 over the completed Hecke algebra $\TT_{\cI_\fn}$, by Theorem \ref{thm:gorenstein} (or Theorem \ref{thm:hidatheory}, respectively), so the tensor product is free over $\Lambda_{\fn}^{\fP}$.
   \end{proof}

   \begin{theorem}
    \label{thm:cmrep}
    For any modulus $\fn$ divisible by $\ff$, the module $H^1(\psi, \fn, \fP)[1/p]$ is isomorphic as a $\Lambda_{\fn}^{\fP}[1/p]$-linear representation of $\Gal(\Qb/\QQ)$ to the induced representation $\Ind_{K(\fn)}^\QQ\left(E(\psi_{\fP})\right)^*$, where $K(\fn)$ is the largest abelian $p$-extension of $K$ of conductor dividing $\fn$ (i.e.~the ray class field corresponding to $H_{\fn}^{(p)}$).
   \end{theorem}

   \begin{proof}
    This statement is unaffected by enlarging $L$, so we may assume $L_{\fP}$ is sufficiently large that all characters $H_{\fn}^{(p)} \to \Qpb^\times$ take values in $L_{\fP}$. Then $\Lambda_{\fn}^{\fP}[1/p] = L_{\fP}[H_{\fn}^{(p)}]$ is a product of copies of $L_{\fP}$, indexed by the characters of $H_{\fn}^{(p)}$; so it suffices to check that for $\eta$ such a character, the $L_\fP$-vector space
    \begin{equation}
     \label{eq:vspace}
     L_{\fP} \otimes_{\Lambda_{\fn}^{\fP}, \eta} H^1(\psi, \fn, \fP)[1/p]
    \end{equation}
    is 2-dimensional and isomorphic to the $\eta$-isotypical component of $\Ind_{K(\fn)}^\QQ V_{L_{\fP}} \left(\psi\right)^*$, which is $\Ind_{K}^\QQ V_{L_{\fP}}(\psi\eta)^*$.

    However, this vector space \eqref{eq:vspace} is the maximal quotient of $H^1_{\et}(\overline{Y_1(N)}, L_{\fP}(1))$ on which the covariant Hecke operators act via the character of $\TT_{N}$ corresponding to the level $N$ eigenform
    \[ g_{\psi \eta} \coloneqq \sum_{\fa: (\fa, \fn) = 1} \psi(\fa)\eta(\fa)q^{N(\fa)}.\]
    By the multiplicity one theorem, the corresponding quotient of $H^1_{\et}(\overline{Y_1(N)}, L_{\fP})$ is 2-dimensional, and realizes the Galois representation $V_{L_\fP}(\overline{g_{\psi \eta}})$ attached to the complex conjugate eigenform $\overline{g_{\psi \eta}}$. Since we have $V_{L_\fP}(\overline{g_{\psi \eta}})(1) = V_{L_\fP}(g_{\psi \eta})^* = \Ind_{K}^\QQ(\psi \eta)^*$ we are done.
   \end{proof}

   \begin{proposition}
    \label{prop:cmnorm}
    Suppose either that $p$ is inert, $p \nmid \fn$, and $\fl$ is a prime not equal to $p$; or that $p$ is split, $\fpb \nmid \fn$ and $\fl \ne \fpb$. Then the norm map
    \[ \mathcal{N}_{\fn}^{\fl\fn}: \Lambda_{\fn}^{\fP} \otimes_{\Lambda_{\fl \fn}^{\fP}} H^1(\psi, \fl \fn, \fP) \rTo H^1(\psi, \fn, \fP)\]
    is a bijection.
   \end{proposition}

   \begin{proof}
    We assume first that $(p, \fn\fl) = 1$. Since both source and target of the map concerned are free $\Zp$-modules, and the map $\mathcal{N}_{\fn}^{\fl\fn}$ is an isomorphism after inverting $p$, it suffices to check that it is surjective. As before, let $N = N_{K/\QQ}(\fn) \cdot \disc$ and $N' = N_{K/\QQ}(\fn\fl) \cdot \disc$, and let $\ell$ be the rational prime below $\fl$.

    If $\fl \mid \fn$, then $\mathcal{N}_{\fn}^{\fl\fn}$ is the map induced by $(\pr_1)_*: H^1(Y_1(N')) \to H^1(Y_1(N))$, and this is evidently surjective. (Indeed, since $\ell \mid N$, the map $\pr_1: Y_1(N') \to Y_1(N)$ has degree either $\ell$ or $\ell^2$, neither of which is divisible by $p$, so $(\pr_1)_* (\pr_1)^*$ is even surjective.)

    Hence we may assume $\fl \nmid \fn$. There are three cases to consider. Firstly, if $\fl$ is a ramified prime, or if $\fl$ is split and $\overline\fl \mid \fn$, then $\ell \mid N$ and $N' = \ell N$. In this case, comparing Theorem \ref{thm:normmaps2} and Definition \ref{def:curlyN}, we see that $\mathcal{N}_{\fn}^{\fl\fn}$ is the map deduced from the map $\beta$ of Theorem \ref{thm:normmaps2} via base extension along the map $\phi_{\fn}: \TT_{N} \to \Lambda_{\fn}^{\fP}$ of Proposition \ref{prop:defphin} (mapping $U_\ell$ to $\psi(\fl) [\fl]$).

    If $\fl$ is a split prime and $\overline\fl \nmid \fn$, then $\ell \nmid N$, and we apply Theorem \ref{thm:normmaps1} instead. We extend the map $\phi_{\fn}: \TT_{N} \to \Lambda_{\fn}^{\fP}$ to $\widetilde \TT_{N}$ by mapping $X$ to $\psi(\overline\fl)[\overline\fl]$. Since this is also the image of $U_\ell$ under $\phi_{\fl \fn}$, and $T_\ell - X$ maps to $\psi(\fl)[\fl]$, the map $\beta$ of Theorem \ref{thm:normmaps1} again gives rise to $\mathcal{N}_{\fn}^{\fl\fn}$.

    Finally, if $\fl = (\ell)$ is inert in $K$ then we apply Corollary \ref{cor:normmaps3}, and the calculation proceeds similarly, using the fact that $\phi_{\fn}$ maps $T_\ell$ to 0 and $\langle \ell \rangle$ to $-\tfrac{\psi(\ell)[\ell]}{\ell}$.

    If $p$ is split, $\fp \mid \fn$, and $\fl \ne \fp$, then we argue similarly using Proposition \ref{prop:normmapshida} in place of Theorems \ref{thm:normmaps1} and \ref{thm:normmaps2}, using the fact that $\cI_{\fn}$ is $p$-distinguished. If $\fl = \fp$ and $\fp \mid \fn$, then the result is immediate from Ohta's control theorem; and if $\fl = \fp$ and $\fp \nmid \fn$, we use Proposition \ref{prop:normmapshidap}.
   \end{proof}

   \begin{corollary}
    \label{cor:nui}
    Let $A$ be the set of ideals of $K$ coprime to $\fp$ and divisible by $\ff$. Then we may find a family of isomorphisms
    \[ \nu_\fn: H^1(\psi, \fn, \fP) \rTo^\cong \Ind_{K(\fn)}^{\QQ} \cO(\psi_{\fP}^{-1}) \]
    of $\Lambda_{\fn}^{\fP}[\Gal(\overline{\QQ} / \QQ)]$-modules, for all $\fn \in A$, with the property that for any two moduli $\fn, \fn' \in A$ with $\fn \mid \fn'$, the diagram
    \begin{diagram}
     H^1(\psi, \fn', \fP) & \rTo_{\cong}^{\nu_{\fn'}} & \Ind_{K(\fn')}^{\QQ} \cO(\psi_{\fP}^{-1})\\
     \dTo^{\mathcal{N}^{\fn'}_{\fn}} & & \dTo \\
     H^1(\psi, \fn, \fP) & \rTo_{\cong}^{\nu_{\fn}} & \Ind_{K(\fn)}^{\QQ} \cO(\psi_{\fP}^{-1})
    \end{diagram}
    commutes.
   \end{corollary}

   \begin{proof}
    Firstly, let $(\fn_i)_{i \ge 1}$ be a sequence of ideals in $A$ such that
    \begin{itemize}
     \item $\fn_1 = \ff$,
     \item $\fn_{i+1} = \fl_i \fn_i$ for all $i \ge 1$, where $\fl_i$ is prime,
     \item every $\fn \in A$ divides $\fn_i$ for some $i \gg 0$.
    \end{itemize}

    Let $A_i$ be the finite set $\{ \fn \in A: \fn \mid \fn_i\}$. Since $\bigcup_{i \ge 1} A_i = A$, it suffices to show that for each $i \ge 1$, there exists a system of isomorphisms $\nu_{\fn}$ for $\fn \in A_i$ such that the compatibility diagram commutes when $\fn, \fn' \in A_i$. We shall prove this claim by induction on $i$.

    We let $\nu_{\ff}$ be any choice of isomorphism
    \[ H^1(\psi, \ff, \fP) \cong \Ind_{K(\ff)}^\QQ \cO(\psi_{\fP}^{-1}), \]
    (which exists by Proposition \ref{prop:cmfree} and Theorem \ref{thm:cmrep}). As $A_1 = \{\ff\}$, this proves our claim for $i = 1$.

    Now suppose that $\nu_{\fn}$ is defined for all $\fn \mid \fn_i$. Let $\nu'$ be any choice of isomorphism $H^1(\psi, \fn_{i+1}, \fP) \cong \Ind_{K(\fn_{i+1})}^{\QQ} \cO(\psi_{\fP}^{-1})$ (which exists, again, by Proposition \ref{prop:cmfree} and Theorem \ref{thm:cmrep}). There is a unique $a \in \Lambda_{\fn_i}^{\fP}$ such that the isomorphism $H^1(\psi, \fn_{i}, \fP) \cong \Ind_{K(\fn_i)}^{\QQ} \cO(\psi_{\fP}^{-1})$ induced by $\nu'$ is equal to $a \cdot \nu_{\fn_i}$.

    Since the morphism $(\Lambda_{\fn_{i+1}}^{\fP})^\times \to (\Lambda_{\fn_{i}}^{\fP})^\times$ is surjective, we can choose a lifting $b$ of $a$ to $\Lambda_{\fn_{i+1}}^{\fP}$, and define $\nu_{\fn_{i+1}} = b^{-1} \nu'$.

    We now define $\nu_{\fn}$, for any $\fn \in A_{i+1}$, to be the morphism induced by $\nu_{\fn_{i+1}}$. This agrees with the existing definition of $\nu_{\fn}$ for $\fn \in A_{i} \subset A_{i+1}$, and the diagram now commutes for all $\fn, \fn' \in A_{i+1}$ as required.
   \end{proof}

  \subsection{\'Etale cohomology classes}

   We now bring the eigenform $f$ back into the picture. We assume henceforth (largely for convenience) that $p \nmid N_f$.

   Recall the motivic cohomology space $H^3_{\mot}(f, \psi, m, \fn, \fO_L(2))$ constructed above. The \'etale regulator map
   \begin{multline*}
    \operatorname{reg}_{\et}: H^3_{\mot}(Y_1(N_f) \times Y_1(N) \times \Spec \QQ(\mu_m), \ZZ(2)) \\\to H^1\left( \QQ(\mu_m), H^1_{\et}(\overline{Y_1(N_f)}, \Zp(1)) \otimes_{\Zp} H^1_{\et}(\overline{Y_1(N)}, \Zp(1))\right)\end{multline*}
   is compatible with correspondences, and therefore descends to a map
   \[ H^3_{\mot}(f, \psi, m, \fn, \fO_L(2)) \to H^1\left(\QQ(\mu_m), T_{\cO}(f)^* \otimes H^1(\psi, \fn, \fP)\right)\]
   where $T_{\cO}(f)^*$ is the quotient of $H^1_{\et}(\overline{Y_1(N_f)}, \cO(1))$ defined as in \cite{leiloefflerzerbes13}, and $H^1(\psi, \fn, \fP)$ is as defined above.

   We now choose a set of isomorphisms $\{\nu_{\fn}: \fn \in A\}$ as in Corollary \ref{cor:nui}. By Shapiro's lemma, we have a canonical isomorphism
   \[ H^1(\QQ, \Ind_{K(\fn)}^{\QQ} \cO(\psi^{-1})) \cong H^1(K(\fn), \cO(\psi^{-1}).\]

   \begin{definition}
    For $c > 1$ coprime to $6 N_f N_\psi$, and $\fn \in A$, let
    \[ {}_c \bfz_{\fn}^{f, \psi} \in H^1\left(K(\fn), T_{\cO}(f)^*(\psi^{-1})\right)\]
    be the image of ${}_c \Xi_{1, \fn}^{f, \psi}$ under the above map.

    If $\fn$ is an ideal coprime to $\fp$, but not divisible by $\ff$, we define ${}_c \bfz_{\fn}^{f, \psi}$ as the image under corestriction of ${}_c \bfz_{\fn \ff}^{f, \psi}$.
   \end{definition}

   We first show that we may get rid of the factor $c$. Let $\varepsilon = \varepsilon_f \cdot \chi \cdot \varepsilon_K$ be the product of the Nebentypus characters of $f$ and $g_\psi$. Let us write $N_\psi = N_{K/\QQ}(\ff) \cdot \disc$, which is coprime to $p$.

   We know that if $\fn$ is divisible by $\ff$ and $c, d$ are two integers $> 1$ and coprime to $6 N_f N$, where $N = N_{K/\QQ}(\fn) \cdot \disc$ as usual, then
   \[ (c^2 - \varepsilon(c)^{-1}[c]^{-2}) {}_d \bfz_{\fn}^{f, \psi} \]
   is symmetric in $c$ and $d$ (cf.~\cite[Proposition 2.7.5(5)]{leiloefflerzerbes13}).

   Since $p > 3$ and $p$ does not divide $N_f N_\psi$, there exists some $d > 1$ such that $d^2 \ne 1 \bmod p$ and $d = 1 \bmod N_f N_\psi$. We may also assume that $d$ is coprime to $6 N$. We have $\varepsilon(d) = 1$, so $d^2 - \varepsilon(d)^{-1}[d]^{-1}$ is invertible in $\Lambda_{\fn}^{\fP}$; and if we define
   \[ \bfz^{f, \psi}_{\fn} = (d^2 - \varepsilon(d)^{-1}[d]^{-1})^{-1} {}_d \bfz^{f, \psi}_{\fn} \in H^1(K(\fn), T),\]
   then $\bfz^{f, \psi}_{\fn}$ is independent of $d$ and we have ${}_c \bfz^{f, \psi}_{\fn} = (c^2 - \varepsilon(c)^{-1}[c]^{-2})\bfz^{f, \psi}_{\fn}$ for any valid choice of $c$.

   \begin{theorem}
    \label{thm:eulersystem}
    Let $\mathcal{N} = p N_f \ff$. Then the elements
    \[ \{ \bfz_{\fn}^{f, \psi}: (\fn, \mathcal{N}) = 1)\}\]
    form an Euler system for $(T, \mathcal{K}, \mathcal{N})$ in the sense of \cite{rubin00}, where $\mathcal{K}$ is the composite of the $K(\fn)$ for all $\fn$ coprime to $\mathcal{N}$.
   \end{theorem}

   \begin{proof}
    By Theorem \ref{thm:normcompat} (and the compatibility of the \'etale regulator with correspondences), these elements satisfy the Euler system compatibility relation.
   \end{proof}

  \subsection{Local properties}

   We now show that the classes $\bfz_{\fn}^{f, \psi}$ have good local behaviour. We recall the following definition, due to Bloch and Kato \cite{blochkato90}:

   \begin{definition}
    If $V$ is any continuous $\Qp$-linear representation of $\Gal(\overline{L} / L)$, where $L$ is a finite extension of $\QQ_\ell$, the Bloch--Kato Selmer subspace $H^1_f(L, V)$ is defined as follows (cf.~\cite{blochkato90}):
    \begin{itemize}
     \item if $\ell \ne p$, we define $H^1_f(L, V) = H^1(L^{\operatorname{nr}} / L, V^{I_L})$, where $I_L$ is the inertia subgroup of $\Gal(\overline{L} / L)$;
     \item if $\ell = p$, we define $H^1_f(L, V) = \operatorname{ker}\left(H^1(L, V) \to H^1(L, V \otimes \mathbf{B}_{\mathrm{cris}})\right)$ where $\mathbf{B}_{\mathrm{cris}}$ is Fontaine's crystalline period ring.
    \end{itemize}
    For $T \subseteq V$ a $\Gal(\overline{L} / L)$-stable $\Zp$-lattice, we define $H^1_f(L, T)$ and $H^1_f(L, V/T)$ as the preimage (resp.~image) of $H^1(L, V)$.
   \end{definition}

   \begin{notation}
    For convenience we will use the shorthand $T := T_{\cO}(f)^*(\psi_\fP^{-1})$.
   \end{notation}

   \begin{proposition}
    \label{prop:h1f}
    Suppose that one of the following conditions holds:
    \begin{enumerate}[(i)]
     \item $p$ is split in $K / \QQ$, and the polynomial
     \[ P_p\left(\frac{\psi(\fp)}{p} X\right) \]
     does not vanish at any $p$-power root of unity.
     \item $p$ is inert in $K / \QQ$ and $v_{\fP}(a_p(f)) < \tfrac 1 2$.
    \end{enumerate}
    Then for any $\fn$ coprime to $\mathcal{N}$, and any prime $v \nmid p$ of $K(\fn)$, we have
    \[ \loc_v\left(\bfz_{\fn}^{f, \psi}\right) \in H^1_f(K(\fn)_v, T).\]
   \end{proposition}

   \begin{proof}
    In case (i), to show that $\bfz_{\fn}^{f, \psi}$ lies in the local $H^1_f$, we compare it with the class $\bfz_{\fn\fp}^{f, \psi}$. We know that $\bfz_{\fn\fp}^{f, \psi}$ is a universal norm from the tower $K(\fn \fp^\infty) / K(\fn\fp)$, which is a $\Zp$-extension in which no finite prime splits completely. Hence it is automatically in $H^1_f$ at all primes away from $p$, by \cite[Corollary B.3.5]{rubin00}. However, we have
    \[ \mathcal{N}_{\fn}^{\fn\fp}\left(\bfz_{\fp\fn}^{f, \psi}\right) = P_{p}\left(\frac{\psi(\fp)}{p}[\fp] \right)  \bfz^{f, \psi}_{\fn}.\]
    If no root of $P_{p}\left(\frac{\psi(\fp)}{p} X\right)$ is a root of unity of order dividing $\# H_{\fn}^{(p)}$, the element $P_{p}\left(\frac{\psi(\fp)}{p} [\fp]\right)$ is a unit in $\Lambda_{\fn}^{\fP}[1/p]$; but the action of $\Lambda_{\fn}^{\fP}[1/p]$ preserves $H^1_f$, so we are done.

    In case (ii), we use Corollary 6.7.9 of \cite{leiloefflerzerbes13}. This shows that for $f, g$ of level prime to $p$, the class $\bfz^{f, g}_{1}$ is in $H^1_f$ if there exist $p$-stabilizations $\alpha$ of $f$ and $\gamma$ of $g$ such that $v_p(\alpha \gamma) < 1$ and none of the elements
    \[ \left\{ \alpha\gamma, \frac{\alpha\delta}{p}, \frac{\beta \gamma}{p}, \frac{\beta\delta}{p}\right\}. \]
    are equal to 1. We apply this with $g = g_{\psi\eta}$ for each character $\eta$ of $H_{\fn}^{(p)}$; then we have $v_\fP(\alpha) < \tfrac12$, $v_{\fP}(\beta) > \tfrac12$, and $v_{\fP}(\gamma) = v_{\fP}(\delta) = \tfrac12$, so none of these four quantities can be a $\fP$-adic unit.
   \end{proof}

   \begin{remark}
    We take the opportunity to note that there is a small gap in the proof of Proposition 6.6.2 of \cite{leiloefflerzerbes13} (on which the cited corollary 6.7.9 relies). The argument actually only proves that $\bfz^{f, g}_{1}$ is in $H^1_f$ if $\alpha_f \alpha_g \ne 1$, since $\bfz^{f,g}_1 = (\alpha_f \alpha_g - 1) \norm_{\QQ}^{\QQ(\mu_p)} \bfz^{f,g}_p$. It can actually happen that $\alpha_f \alpha_g = 1$ (e.g.~if $f$ and $g$ both correspond to elliptic curves with split multiplicative reduction at $p$). However, we are interested in the case when $f$ and $g$ are the $p$-stabilizations of eigenforms of level prime to $p$, in which case $\alpha_f \alpha_g$ is a Weil number of weight 2, and thus cannot be equal to 1.
   \end{remark}

   We now consider the local properties of our Euler system at the primes of $K$ above $p$. This is very straightforward\footnote{Straightforward, that is, modulo the rather deep fact that the \'etale regulator maps classes in the $K$-theory of a smooth proper $\Zp$-scheme to classes in $H^1_f$.} from the construction of the Beilinson--Flach elements.

   \begin{proposition}
    \label{prop:h1fp}
    If $(p, \fn) = 1$, then we have
    \[ \loc_w\left(\bfz_{\fn}^{f, \psi}\right) \in H^1_f\left(K(\fn)_w, T\right)\]
    for all primes $w \mid p$ of $K(\fn)$.
   \end{proposition}

   \begin{proof}
    It suffices to check the result for $\fn = 1$ with $\psi$ replaced by $\psi \eta$, for each character $\eta$ of $H_{\fn}^{(p)}$. This is immediate from \cite[Proposition 6.5.4]{leiloefflerzerbes13} applied to the modular forms $f$ and $g_{\psi\eta}$, which both have level coprime to $p$.
   \end{proof}

 \section{P-adic L-functions}

  We now collect some results on $p$-adic $L$-functions attached to $f$ over $K$. We shall assume throughout that $f$ does not have CM by $K$, so the base-change of $f$ to an automorphic representation of $\GL_2(\AA_K)$ is cuspidal.

  \subsection{Definition of the $L$-functions}
   \label{sect:defLfcn}
   Let $\Psi$ be any $L$-valued algebraic Gr\"ossen\-character of $K$, of some arbitrary infinity-type $(a,b)$. We write $L(f/K, \Psi, s)$ for the $L$-function attached to the base-change of $f$ to $K$ twisted by $\Psi$. Then the point $s = 1$ is a critical value of the $L$-function $L(f/K, \Psi, s)$ if and only if one of the following holds:
   \begin{itemize}
    \item we have $a = b = 0$ (region $\Sigma^{(1)}$);
    \item we have $a \le -1$ and $b \ge 1$ (region $\Sigma^{(2)}$);
    \item we have $b \le -1$ and $a \ge 1$ (region $\Sigma^{(2')}$);
   \end{itemize}
   See Figure \ref{fig:hecke} below.

   \begin{remark} Our notation for the critical regions is taken from Definition 4.1 of \cite{BDP13}, but our conventions are slightly different, since we work with $L(f/K, \Psi, 1)$ rather than $L(f/K, \Psi^{-1}, 0)$. Thus our Figure 1 is Figure 1 of \cite{BDP13} rotated by $180^\circ$ around the point $(\tfrac{1}{2}, \tfrac{1}{2})$.
   \end{remark}

   \begin{figure}[ht!]
    \includegraphics[width=\textwidth]{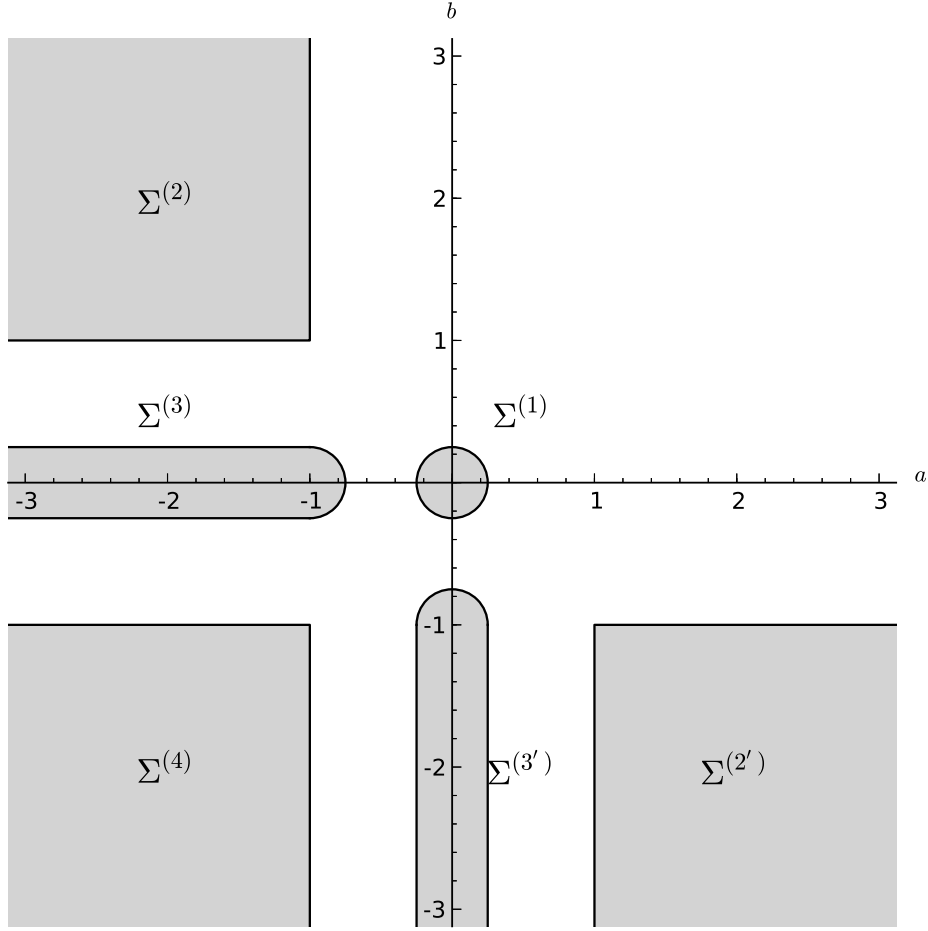}
    \caption{\label{fig:hecke}Infinity-types of Hecke characters of $K$}
   \end{figure}

   The regions $\Sigma^{(3)}$, $\Sigma^{(3')}$, and $\Sigma^{(4)}$ in Figure \ref{fig:hecke} correspond to characters where the archimedean $\Gamma$-factor $L_\infty(f/K, \Psi, s)$ has a pole at $s = 1$, of order 1 for $\Sigma^{(3)}$ and $\Sigma^{(3')}$, and of order 2 for $\Sigma^{(4)}$. Since the completed $L$-function $\Lambda(f/K, \Psi,  s) = L(f/K, \Psi, s)L_\infty(f/K, \Psi, s)$ is holomorphic on $\CC$ and nonzero\footnote{If $a - b > 1$ then $s = 1$ is in the region of convergence of the Euler product and thus the $L$-value is nonzero; the case $a - b < - 1$ follows from this via the functional equation. The remaining cases $a - b = \pm 1$ follow from a deep global non-vanishing statement of Jacquet and Shalika \cite{jacquetshalika76}.} at $s = 1$ whenever the $\infty$-type of $\Psi$ does not lie on the at line $b = -a$, it follows that $L(f/K, \Psi, s)$ must vanish at $s = 1$ to order exactly 1 for $\Psi \in \Sigma^{(3)} \cap \Sigma^{(3')}$ and to order exactly 2 for $\Psi \in \Sigma^{(4)}$.

   \begin{remark}
    Beilinson's conjecture \cite{beilinson84} predicts that the vanishing of the $L$-value $L(f/K, \Psi, 1)$ is related to the existence of classes in a motivic cohomology group $H^1_f(K, M_f^*(\psi^{-1}))$, where $M_f$ is the motive of $f$. When $(a, b) = (-1, 0)$ or $(0, -1)$, the conjecture predicts that the motivic cohomology group should be 1-dimensional, and spanned by the Beilinson--Flach classes. It seems reasonable to expect\footnote{Since this paper was originally written, this predicted extension of the construction has been carried out in the paper \cite{kingsloefflerzerbes14}.} that the construction of Beilinson--Flach classes can be generalized to any $(a, b) \in \Sigma^{(3)} \cup \Sigma^{(3')}$. When $(a, b) \in \Sigma^{(4)}$ the motivic cohomology should be 2-dimensional, and Beilinson's conjecture predicts the existence of classes in the group $\bigwedge^2 H^1_f(K, M_f^*(\psi^{-1}))$, but constructing such classes appears to be beyond the reach of present techniques.
   \end{remark}

   We now interpolate $p$-adically, where as above $p$ is a prime $\ge 5$ unramified in $K$ and not dividing $N_f$. Let $\ff$ be a modulus of $K$ with $(p, \ff) = 1$. The ray class group $H_{\ff p^\infty}$ is a $p$-adic analytic group, and algebraic Gr\"ossencharacters of $K$ of conductor dividing $\ff p^\infty$ correspond bijectively with locally algebraic $\overline{\Qp}$-valued characters of $H_{\ff p^\infty}$.

   \begin{theorem}
    Assume $(N_f, N_{\psi}) = 1$, where $N_\psi = N_{K/\QQ}(\ff) \cdot \disc$ as usual, and let $N$ be an integer divisible by $N_\ff N_{\psi}$ and having the same prime factors as $N_\ff N_{\psi}$. Let $\alpha, \beta$ be the roots of the Hecke polynomial of $f$.
    \begin{enumerate}
     \item Suppose $f$ is ordinary at $p$ and $\alpha$ is the unit root. Then there exists an element $L_{\fP}(f/K, \Sigma^{(1)}) \in \Lambda_E(H_{\ff p^\infty})$ with the property that for Gr\"ossencharacters $\psi$ of $K$ of conductor dividing $\ff$ and infinity-type in $\Sigma^{(1)}$, we have
     \[ L_{\fP}(f/K, \Sigma^{(1)})(\psi_{\fP}) = \frac{\mathcal{E}(f, \psi, 1)}{\left(1 - \frac{\beta}{\alpha}\right)\left(1 - \frac{\beta}{p\alpha}\right)} \cdot \frac{i N}{8\pi^2 \langle f, f \rangle_N} \cdot L(f/K, \psi, 1),\]
     where $\mathcal{E}(f, \psi, 1)$ is given by
     \[  \mathcal{E}(f, \psi, 1) =
      \begin{cases}
       \prod_{v \mid p}(1 - p^{-1} \beta \psi(v))(1 - \alpha^{-1} \psi(v)^{-1}) & \text{if $p$ is split},\\
       (1 - p^{-2} \beta^2 \psi(p))(1 - \alpha^{-2} \psi(p)^{-1}) & \text{if $p$ is inert.}
      \end{cases}
     \]
     \item Suppose $p$ is split in $K$. Then there exists an element $L_{\fP}(f/K, \Sigma^{(2)}) \in \operatorname{Frac} \Lambda_E(H_{\ff p^\infty})$ with the property that for Gr\"ossencharacters $\psi$ of $K$ of conductor dividing $\ff$ and infinity-type $(a,b) \in \Sigma^{(2)}$, we have
     \[ L_{\fP}(f/K, \Sigma^{(2)})(\psi_{\fP}) = \frac{\mathcal{E}(\psi, f, 1)}{\left(1 - \frac{\psi(\fp)}{\psi(\fpb)}\right) \left(1- \frac{\psi(\fp)}{p \psi(\fpb)}\right)} \cdot \frac{2^{a-b} i^{b-a-1} b! (b-1)! N^{a + b + 1}}{(2\pi)^{1 + 2b} \langle g_{\lambda}, g_{\lambda}\rangle_N} \cdot L(f/K, \psi, 1),\]
     where the factor
     $\mathcal{E}(\psi, f, 1)$ is given by
     \[  \mathcal{E}(\psi, f, 1) =
      (1 - p^{-1} \psi(\fp) \alpha)(1 - p^{-1} \psi(\fp) \beta)(1 - \psi(\fpb)^{-1} \alpha^{-1}) (1 - \psi(\fpb)^{-1} \beta^{-1}),
     \]
     and $g_{\lambda}$ is the CM eigenform of level $N_\psi$ and weight $1 - a + b \ge 3$ corresponding to the Gr\"ossencharacter $\lambda = \psi |\cdot|^{-b}$ of $\infty$-type $(a-b, 0)$.
    \end{enumerate}
   \end{theorem}

   We give a brief sketch of the proof below, since it will be important for our purposes to know how these $L$-functions are related to the $p$-adic Rankin--Selberg $L$-functions considered in \cite{leiloefflerzerbes13}. We shall not need to consider the case $i = 2'$ explicitly, since complex conjugation interchanges the critical regions $\Sigma^{(2)}$ and $\Sigma^{(2')}$.

  \subsection{The case $i = 1$}

   We consider the formal $q$-expansion
   \[ \Theta = \sum_{\fa: (\fa, \ff p) = 1} [\fa] q^{N(\fa)} \in \Lambda(H_{\ff p^\infty})[[q]].\]
   We can regard this as a $q$-expansion with coefficients that are functions on the formal scheme $\mathcal{W} = \operatorname{Spf} \Lambda(H_{\ff p^\infty})$ parametrizing characters of $H_{\ff p^\infty}$.

   We choose an integer $N$ coprime to $p$ and divisible by $N_f$ and by $N_\psi = N_{K/\QQ}(\ff) \cdot \disc$. For each $\alpha \in \frac{1}{N}\ZZ / \ZZ$, we can consider the family $\Xi^{\ord, p}_{\alpha}(\Sigma^{(1)}, -)$ of $q$-expansions over $\mathcal{W}$ given by
   \[ \Xi^{\ord, p}_{\alpha}(\Sigma^{(1)}, \omega) = e_{\ord}\left[\mathcal{E}_{\alpha}(\omega_{\QQ}^{-1}, 0) \Theta(\omega)\right],\]
   where $\mathcal{E}_{\alpha}(\phi_1, \phi_2)$ is the family of $p$-depleted Eisenstein series over $\Spec \Lambda(\Zp^\times)^2$ defined in \S 5 of \cite{leiloefflerzerbes13}, and $\omega_{\QQ}$ denotes the measure on $\Zp^\times$ obtained by composing $\omega$ with the map $\Zp^\times \into (\fO_K \times \Zp)^\times \to H_{\ff p^\infty}$.

   This defines a measure $\Xi^{\ord, p}_{\alpha}(\Sigma^{(1)})$ on $H_{\ff p^\infty}$ with values in the finite-dimensional $E$-vector space $S_2(\Gamma_1(N) \cap \Gamma_0(p), E)^{\ord}$. We define the $p$-adic $L$-function $L_\fP(f/K, \Sigma^{(1)}) \in \Lambda_E(H_{\ff p^\infty})$ by
   \[ L_\fP(f/K, \Sigma^{(1)}) = \frac{ \left\langle (f^*)^{(p)}, \Xi^{\ord, p}_{1/N}(\Sigma^{(1)}) \right\rangle_{N,p} }
   {\left\langle(f^*)^{(p)}, (f^*)^{(p)}\right\rangle_{N,p}},\]
   where $\langle, \rangle_{N,p}$ denotes the Petersson scalar product at level $\Gamma_1(N_f) \cap \Gamma_0(p)$ (normalized to be conjugate-linear in the first variable and linear in the second), $f^*$ denotes the complex conjugate of $f$, and $(f^*)^{(p)}$ its ordinary $p$-stabilization (whose $U_p$-eigenvalue is $p \beta^{-1}$). It is clear by construction that the $p$-adic Rankin--Selberg $L$-value $\mathcal{D}_{\fP}(f, g_{\psi}, 1/N, 1)$ considered in our previous work is given by
   \[ \mathcal{D}_{\fP}(f, g_{\psi}, 1/N, 1) = L_\fP(f/K, \Sigma^{(1)})(\psi_{\fP}).\]

   On the other hand, the specialization of the family $\Theta$ at a finite-order character $\eta$ of $H_{\ff}$ is the $p$-stabilization of the classical weight 1 theta series corresponding to $\eta$. Applying Proposition 5.4.2 of \cite{leiloefflerzerbes13} gives a formula for $L_\fP(f/K, \Sigma^{(1)})(\eta)$ in terms of the critical $L$-value $L(f/K, \psi, 1)$, which simplifies to the formula stated in the theorem above.

   \begin{remark}
    Computing the value of $L_\fP(f/K, \Sigma^{(1)})$ at a finite-order character $\eta$ which may be ramified at the primes above $p$ is clearly possible in principle, but the calculations involved are unpleasant and messy. See \cite{perrinriou88} for a closely related computation.
   \end{remark}

  \subsection{The case $i = 2$}

   In this case we replace $\Theta$ by the $p$-adic family of \emph{ordinary} theta series indexed by $\Lambda(H_{\ff \fp^\infty})$, given by the formal $q$-expansion
   \[ \bfg = \sum_{(\fa, \ff \fp) = 1} [\fa] \psi(\fa) q^{N(\fa)} \in \Lambda(H_{\ff \fp^\infty})[[q]].\]
   We can write any character of $H_{\ff p^\infty}$ uniquely in the form $\lambda \mu$ where $\lambda$ factors through $H_{\ff p^\infty}$ and $\mu$ factors through the norm map $H_{\ff p^\infty} \to H_{p^\infty} \to \Zp^\times$. We define a measure on $H_{\ff p^\infty}$, with values in $p$-adic ordinary modular forms of tame level $N$, by
   \[ \Xi^{\ord, p}_{\alpha}(\Sigma^{(2)}, \lambda \mu) = e_{\ord}\left[ \mathcal{E}_{\alpha}(\mu - 1, -1 - \lambda_{\QQ} - \mu) \cdot f\right].\]
   Note that the weight-character of $\Xi_{\alpha}(\lambda \mu)^{\ord, p}$ at $p$ is $1-\lambda_{\QQ}$, which is the same as that of the specialization $\bfg(\lambda)$ of the family $\bfg$ at $\lambda$. The theory of $p$-adic interpolation of Petersson products thus gives us a $p$-adic $L$-function $L_{\fP}(f/K, \Sigma^{(2)})\in \operatorname{Frac} \Lambda_E(H_{\ff p^\infty})$ satisfying
   \[ L_{\fP}(f/K, \Sigma^{(2)})(\lambda\mu) = \frac{\left\langle \bfg({\lambda})^*, \Xi_{1/N}(\lambda \mu)^{\ord, p}\right\rangle }{\langle \bfg(\lambda), \bfg(\lambda)\rangle}.\]
   On the one hand, it is clear by construction that $L_{\fP}(f/K, \Sigma^{(2)})(\psi_{\fP})$ is the quantity $\mathcal{D}_{\fP}(g_{\psi}, f, 1/N, 1)$ appearing in \cite{leiloefflerzerbes13}.

   On the other hand, if we evaluate $L_{\fP}(f/K, \Sigma^{(2)})$ at a Gr\"ossencharacter $\omega = \lambda \mu$ of infinity-type $(a, b)$ lying in $\Sigma^{(2)}$ and having conductor prime to $p$, then the infinity-type of $\lambda$ is $(a-b, 0)$, while $\mu = |\cdot|^b$. Thus $\bfg(\lambda)$ is the $p$-stabilization of the classical ordinary CM form $g_{\lambda}$ of level $N$ and weight $k = 1 - a + b \ge 3$. Applying \cite[Proposition 5.4.2]{leiloefflerzerbes13} with $f$,$g$, and $j$ replaced by $g_\lambda$, $f$, and $1 + b$, we obtain a formula for $L_{\fP}(f/K, \Sigma^{(2)})(\lambda \mu)$ in terms of the critical $L$-value $L(f/K, \lambda, 1 + b) = L(f/K, \psi, 1)$ which simplifies to the one given above.

  \subsection{Relation to the Euler system classes}

   In \cite[\S 6.10]{leiloefflerzerbes13}, following \cite{darmonrotger12}, we defined -- for any two modular forms $f, g$ of weight 2, CM or otherwise, with $f$ ordinary -- a vector $\eta_f^{\ur} \otimes \omega_g \in \Fil^1 \DdR(V_E(f) \otimes V_E(g))$.

   In our situation, we thus have vectors $\eta_f^{\ur} \otimes \omega_{g_{\psi}}$ (if $f$ is ordinary) and $\eta_{g_{\psi}}^{\ur} \otimes \omega_{f}$ (if $p$ is split), both lying in the space
   $\Fil^1 \DdR(K\otimes \Qp, V^*)$, where $V = V_E(f)^*(\psi^{-1})$ as before.

   Theorem 5.6.4 of \cite{leiloefflerzerbes13}, which is a very slight variation on the main theorem of \cite{BDR12}, now gives the following:

   \begin{theorem}
    \label{thm:Lvalue}
    If $f$ is ordinary, then
    \[ L_{\fP}(f/K, \Sigma^{(1)})(\psi_{\fP}) = -\frac{\mathcal{E}(f, \psi, 1)}{\left(1 - \frac{\beta}{\alpha}\right)\left(1 - \frac{\beta}{p\alpha}\right)} \left\langle \log_{p, V}(\bfz_1^{f,\psi}), \eta_f^{\ur} \otimes \omega_g\right\rangle,\]
    and if $p$ is split, then
     \[ L_{\fP}(f/K, \Sigma^{(2)})(\psi_{\fP}) = -\frac{\mathcal{E}(\psi, f, 1)}{\left(1 - \frac{\psi(\fp)}{\psi(\fpb)}\right) \left(1- \frac{\psi(\fp)}{p \psi(\fpb)}\right)} \left\langle \log_{p, V}(\bfz_1^{f,\psi}), \omega_f \otimes \eta_{g_\psi}^{\ur}\right\rangle.\]
    \end{theorem}

 \section{Bounding Selmer groups}

  \subsection{Big image results}

   In this section, we collect some results we will need regarding the image of $\Gal(\overline{K} / K)$ acting on the representation $T = T_{\cO}(f)^*(\psi_{\fP}^{-1})$ and $V = T[1/p]$. Let $K^{\ab}$ be the maximal abelian extension of $K$.

   We impose the following assumption on $f$, which will be in force for the remainder of this paper:

   \begin{assumption}
    The modular form $f$ is not of CM type.
   \end{assumption}

   Under this assumption, it has been shown by Momose \cite{momose81} that there is a number field $F \subseteq L$, a quaternion algebra $B / F$, and an embedding $B \into M_{2 \times 2}(L)$, such that for any prime $\fP$ of $L$, the image of $G_{\QQ}$ in $\Aut V_{L_{\fP}}(f)) \cong \GL_2(L_{\fP})$ contains an open subgroup of the group
   \[ \{ x \in (B \otimes_F F_{\fP})^\times: \operatorname{norm}(x) \in \Qp^\times \} \]
   (where $F_{\fP}$ denotes the completion of $F$ at the prime below $\fP$, and $\operatorname{norm}$ is the reduced norm map of $B$).

   We now impose a restriction on the prime $\fP$:

   \begin{assumption}
    The quaternion algebra $B$ is unramified at $\fP$, so $(B \otimes_F F_{\fP})^\times = \GL_2(F_{\fP})$.
   \end{assumption}

   \begin{remark}
    Note that $B$ is split over the field generated by the Fourier coefficients of $f$; so if $f$ has rational coefficents, $B$ must be the split algebra and this assumption is automatic. In any case, the set of primes ramified in $B$ is finite, and it has been shown \cite{GhateGJQuer} that the primes ramifying in $B$ are a subset of the primes dividing $2 N_f \operatorname{disc} \QQ(f)$.
   \end{remark}

   \begin{proposition}
    \label{prop:bigimage1} \mbox{~}
    \begin{enumerate}[(i)]
     \item The representation $V$ is irreducible as a representation of $\Gal(\overline{K} / K^{\ab})$.
     \item There exists an element $\tau \in \Gal(\overline{K} / K^{\ab})$ (the derived subgroup of $\Gal(\overline{K} / K)$) such that $V / (\tau - 1) V$ is 1-dimensional.
     \item There exists an element $\gamma \in \Gal(\overline{K} / K^{\ab})$ such that $V^{\gamma = 1} = 0$.
    \end{enumerate}
   \end{proposition}

   \begin{proof}
    Because of our two assumptions above, Momose's theorem shows that there is an $L_{\fP}$-basis of $V_{L_{\fP}}(f)$ such that the image of $\Gal(\overline{\QQ} / \QQ)$ in $\GL_2(L_{\fP})$ with respect to this basis contains an open subgroup of $\GL_2(\Zp)$. The subgroup $\Gal(\overline{K} / K)$ is open, so its image also contains an open subgroup of $\GL_2(\Zp)$. However, the derived subgroup of an open subgroup of $\GL_2(\Zp)$ is an open subgroup of $\SL_2(\Zp)$, so the image of $\Gal(\overline{K} / K^{\ab})$ contains an open subgroup of $\SL_2(\Zp)$.

    This certainly implies that $V$ is irreducible restricted to $\Gal(\overline{K} / K^{\ab})$. Moreover, it implies that the image of $\Gal(\overline{K} / K^{\ab})$ in $\Aut_{L_\fP}(V)$ contains an element of the form $\begin{pmatrix} 1 & x \\ 0 & 1 \end{pmatrix}$ with $x \ne 0$; since $\Gal(\overline{K} / K^{\ab})$ acts trivially on the one-dimensional representation $L_{\fP}(\psi)^*$, it follows that a $\tau$ as in (ii) exists.

    Finally, the existence of a $\gamma$ as in (iii) is rather obvious: we may find $y \in \Zp^\times$ with $y \ne 1$, but $y$ sufficiently close to 1 that $\begin{pmatrix} y & 0 \\ 0 & y^{-1}\end{pmatrix}$ is in the image of $\Gal(\overline{K} / K^{\ab})$.
   \end{proof}

   If we impose an additional assumption on $f$ then we have stronger results:

   \begin{notation}
    We say $f$ \emph{has big image at $\fP$} if the image of $\Gal(\overline{\QQ} / \QQ)$ in the group $\Aut T_{\cO}(f)$ contains a conjugate of $\SL_2(\Zp)$.
   \end{notation}

   By a theorem of Ribet \cite{ribet85}, since we are assuming that $f$ is not of CM type, it has big image at almost all primes of $L$.

   \begin{proposition}
    \label{prop:bigimage2}
    Suppose that $f$ has big image at $\fP$. Then
    \begin{enumerate}[(i)]
     \item $T /\fP T$ is irreducible as a representation of $\Gal(\overline{K} / K^{\ab})$.
     \item There exists $\tau \in \Gal(\overline{K} / K^{\ab})$ such that $T / (\tau - 1) T$ is free of rank 1 over $\cO$.
     \item We have
     \[ H^1(\Omega / K, T \otimes_{\Zp} \Qp/\Zp) = H^1(\Omega / K, T^*(1) \otimes_{\Zp} \Qp/\Zp) = 0,\]
     where $\Omega$ is the smallest extension of $K$ containing $K(1) K(\mu_{p^\infty})$ and such that $\Gal(\overline{K} / \Omega)$ acts trivially on $T$.
     \item The $\cO$-module $H^1(K, T)$ is free.
    \end{enumerate}
   \end{proposition}

   \begin{proof}
    For $p > 2$ the group $\SL_2(\Zp)$ has no normal subgroups of index 2. Thus the intersection of the image of $\Gal(\overline{K} / K)$ with the conjugate of $\SL_2(\Zp)$ inside $\Aut T_{\fO_{L, \fP}}(f)$ must be the whole of $\SL_2(\Zp)$. As $\SL_2(\Zp)$ is equal to its commutator subgroup, we deduce that the image of $\Gal(\overline{K} / K^{\ab})$ in $\Aut T$ also contains a conjugate of $\SL_2(\Zp)$. Thus (i) is obvious, and for (ii) we can take $\tau$ to be any element mapping to $\begin{pmatrix} 1 & 1 \\ 0 & 1 \end{pmatrix}$.

    We now prove (iii). Let $\gamma \in \Gal(\overline{K} / K^{\ab})$ be such that $\gamma$ maps to $-1 \in \SL_2(\Zp)$. Then the subgroup $S$ of $G = \Gal(\Omega / K)$ generated by the image of $\gamma$ is in the centre of $G$ and satisfies $H^0(S, T \otimes \Qp/\Zp) = H^1(S, T \otimes \Qp/\Zp) = 0$, and similarly for $T^*(1)$. Via the inflation-restriction exact sequence the required vanishing follows.

    Lastly, we check the freeness statement. From the cohomology long exact sequence arising from $0 \to T \rTo^{\times \varpi} T \to T/\fP T \to 0$, where $\varpi$ is a uniformizer of $\fO_{L, \fP}$, we have a surjection $H^0(K, T/ \fP T) \onto H^1(K^\Sigma / K, T)[\fP]$; but we know that $H^0(K, T/ \fP T) = 0$, so $H^1(K, T)$ is torsion-free and thus free.
   \end{proof}

  \subsection{Selmer groups: definitions}

   We now recall the definitions of some Selmer groups we will need. For this section (only), $K$ may be any number field, and $T$ any $\cO$-linear representation of $\Gal(\overline{K} / K)$ unramified at almost all primes. Let $T^\vee = \Hom_{\cO}(T, E/\cO)$ be the Pontryagin dual of $T$.

   \begin{definition}[{cf.~\cite[Definition 1.5.1]{rubin00}}]
    Let $\Sigma$ be a finite set of places of $K$. We define
    \[ \Sel^{\Sigma}(K, T^\vee(1)) = \ker \left( H^1(K, T^\vee(1)) \to \bigoplus_{v \notin \Sigma} \frac{H^1(K_v, T^\vee(1))}{H^1_f(K_v, T^\vee(1))}\right),\]
    and
    \[ \Sel_{\Sigma}(K, T^\vee(1)) = \ker\left( \Sel^{\Sigma}(K, T^\vee(1)) \to \bigoplus_{v \in \Sigma} H^1(K_v, T ^\vee(1))\right).\]
   \end{definition}

   When $\Sigma$ is the empty set, we simply write $\Sel(K, T^\vee(1))$ for $\Sel^\varnothing(K, T^\vee(1)) = \Sel_\varnothing(K, T^\vee(1))$, the Bloch--Kato Selmer group. We write $\Sigma_p$ for the set of primes of $K$ above $p$.

  \subsection{Bounding the strict Selmer group}

   Let us use the notation
   \[ \bfz^{f, \psi} \in H^1(K, T)\]
   for the image of $\bfz_{1}^{f, \psi}$ under evaluation at the trivial character of $H_{1}^{(p)}$.

   \begin{theorem}[Selmer finiteness]
    \label{thm:finiteSelK}
    Suppose that either
    \begin{itemize}
     \item $p$ is split in $K$ and $P_p\left(\frac{\psi(\fp)}{p}\right) \ne 0$;
     \item or $p$ is inert in $K$ and $v_\fP(a_p(f)) < \tfrac12$.
    \end{itemize}

    If $\bfz^{f, \psi} \ne 0$, then $\Sel_{\Sigma_p}(K, T^\vee(1))$ is finite.
   \end{theorem}

   \begin{proof}
    This follows by applying Theorem 2.2.3 of \cite{rubin00} to our Euler system.

    Suppose we are in the inert case. Let $\Sigma$ be the set of primes of $K$ dividing $\mathcal{N}$, where $\mathcal{N} = p N_f \ff$ as before. Via Theorem \ref{thm:eulersystem} we have an Euler system for $(T, \mathcal{K}, \mathcal{N})$ in which the base class over $K$ is non-torsion. Moreover, Rubin's hypothesis $\Hyp(K, V)$ is satisfied by Proposition \ref{prop:bigimage1}. Our $\mathcal{K}$ does not contain a $\Zp$-extension, but by Proposition \ref{prop:h1f} every class in our Euler system is in $H^1_f$ away from $p$, so we may use the modified version of Theorem 2.2.3 assuming the condition (ii')(b) in \S 9.1 of \emph{op.cit.}; the element $\gamma$ called for in this case is supplied by Proposition \ref{prop:bigimage1}; it is clear that $T^{G_{K(1)}} = 0$, so the modified version of Rubin's Theorem 2.2.3 applies and we deduce that $\Sel_{\Sigma_p}(K, T^\vee(1))$ is finite.

    When $p$ is split we proceed slightly differently: for each $\fn$ coprime to $\mathcal{N}$, we replace $\bfz_{\fn}^{f, \psi}$ with the element
    \[ \hat{\bfz}_{\fn}^{f, \psi} \coloneqq  \mathcal{N}_{\fn}^{\fn \fp}\left( \bfz_{\fn\fp}^{f, \psi}\right).\]
    Our assumption that $P_p\left(\frac{\psi(\fp)}{p}\right) \ne 0$ implies that $\hat{\bfz}^{f, \psi} \ne 0$ if and only if $\bfz^{f, \psi} \ne 0$. Moreover, each class $\hat{\bfz}_{\fn}^{f, \psi}$ is a universal norm from the $\Zp$-extension $K(\fn\fp^\infty) / K(\fn)$, and is therefore in $H^1_f$ locally away from $p$. We now proceed as before.

    (Alternatively, we can replace $\mathcal{K}$ with the compositum $\mathcal{K}'$, of $\mathcal{K}$ and $K(\fp^\infty)$; the $\hat{\bfz}_{\fn}^{f, \psi}$ extend to an Euler system for $(T, \mathcal{K}', \mathcal{N})$, and we can now apply Rubin's theorem 2.2.3 in its original form.)
   \end{proof}

   We now give a bound for the fine Selmer group.

   \begin{theorem}[Bound for Selmer]
    \label{thm:boundedSelK}
    Suppose that the modular form $f$ has big image at $\fP$, and one of the following hypotheses holds:
    \begin{itemize}
     \item $p$ is split in $K$ and no root of $P_p\left(\frac{\psi(\fp)}{p} X\right)$ is a $p$-power root of unity;
     \item $p$ is inert in $K$ and $v_\fP(a_p(f)) < \tfrac12$.
    \end{itemize}

    If $\bfz^{f, \psi}$ is non-torsion, then we have the bound
    \[ \ell_{\cO}\left(\Sel_{\Sigma_p}(K, T^\vee(1))\right) \le \ind_{\cO}\left(\bfz^{f,\psi}\right), \]

    If  $p$ is split in $K$, but $P_p\left(\frac{\psi(\fp)}{p} X\right)$ does have a root that is a $p$-power root of unity, then we have
    \[ \ell_{\cO}\left(\Sel_{\Sigma_p}(K, T^\vee(1))\right) \le  \ind_{\cO}\left(\bfz^{f,\psi}\right) + v_{\fP} P_p\left(\frac{\psi(\fp)}{p}\right).\]
   \end{theorem}

   \begin{proof}
    We now apply Theorem 2.2.2 of \cite{rubin00} rather than Theorem 2.2.3. The additional hypothesis $\Hyp(K, T)$ required in this theorem is supplied by Proposition \ref{prop:bigimage2}, which also shows that the quantities $n_W$ and $n^*_W$ appearing in Rubin's statement are both zero in our setting.

    The first statement corresponds to applying Rubin's theorem to the Euler system for $(T, \mathcal{K}, \mathcal{N})$ as in the proof of the previous theorem. By proposition \ref{prop:h1f}, our slightly stronger assumption on $P_p$ in the split case implies that all the classes in this system are in $H^1_f$ away from $p$.

    If $P_p\left(\tfrac{\psi(\fp)}{p} X\right)$ does have roots that are $p$-power roots of unity, then we instead use the modified Euler system $\hat{\bfz}_{\fn}^{\psi}$ as in the previous proof. We have
    \[ \ind_{\cO} \left(\hat{\bfz}^{f, \psi}\right) = \ind_{\cO} \left(\bfz^{f, \psi}\right) + v_{\fP} P_p\left(\frac{\psi(\fp)}{p}\right)\]
    and this gives the weaker Selmer bound in this case.
   \end{proof}

   \begin{remark}
    If $p$ is split and $P_p\left(\frac{\psi(\fp)}{p}\right) = 0$, then the statement of Theorem \ref{thm:boundedSelK} is still true, but vacuous (the upper bound is $\infty$). This should perhaps be understood as a ``trivial zero'' phenomenon.
   \end{remark}

  \subsection{Bounding the Bloch--Kato Selmer group}

   We now show that the Euler system can also be used to bound the Bloch--Kato Selmer group $\Sel(K, T^\vee(1))$. Sadly we can only do this under very much more restrictive local hypotheses.

   \begin{assumption}
    \label{assump:ord}
    The following conditions are satisfied:
    \begin{enumerate}[(i)]
     \item $p$ is split in $K$.
     \item The modular form $f$ is ordinary at $p$ (i.e. $v_\fP(a_p(f)) = 0$).
     \item We have $\alpha \psi(\fpb) \ne 1 \bmod \fP$ and $\frac{\beta \psi(\fpb)}{p} \ne 1$, where $\alpha$ and $\beta$ are the unit and non-unit roots of the Hecke polynomial of $f$ at $p$.
     \item We have $\frac{\alpha \psi(\fp)}{p} \notin \mu_{p^\infty}$.
    \end{enumerate}
   \end{assumption}

   \begin{theorem}
    \label{thm:finiteBKSel}
    Assume that $f$ is not of CM type and Assumption \ref{assump:ord} holds. Then, if $\bfz^{f, \psi} \ne 0$, the Bloch--Kato Selmer group $\Sel(K, T^\vee(1))$ is finite.

    If in addition $f$ has big image at $\fP$, then we have
    \[ \ell_\cO\left(\Sel(K, T^\vee(1))\right) \le \ind_{\cO} \left(\bfz^{f,\psi}\right).\]
   \end{theorem}

   \begin{proof}
    This follows by applying a modified version of the Euler system machinery which is summarized by Theorem \ref{thm:ESmachine} in Appendix \ref{sect:appendixB} below. So we must prove that the hypotheses of that theorem are satisfied.

    We need to show that for $v = \fp, \fpb$, there is a subspace $V_v^+ \subseteq V$ stable under the decomposition group $D_v$ satisfying the conditions of Corollary \S \ref{cor:kolyvagin}. Recall that, since $f$ is ordinary, there exists a unique one-dimensional unramified subrepresentation $\mathscr{F}^+ V_E(f) \subset V_E(f)$ stable under $D_p$. We define $V_{\fpb}^+$ by
    \[V_{\fpb}^+ = \left(\frac{V_{E}(f)}{\mathscr{F}^+ V_E(f)}\right)^*(\psi_{\fP}^{-1}) \subset V_E(f)^*(\psi^{-1}) = V.\]
    Meanwhile, we define $V_{\fp}^+ = V$. Then for each $v$, the space $V_{v}^+$ is the unique subrepresentation of $V |_{D_{v}}$ such that $V_v^+$ has all Hodge--Tate weights $\ge 1$ and $V / V_v^+$ has all Hodge--Tate weights $\le 0$.

    We set $T_v^+ = V_v^+ \cap T$. I claim that $H^0(K_v, (T/T_v^+) \otimes \mathbf{k}) = 0$. For $v = \fp$ this is selfevident, since $T_{\fp}^+ = T_{\fp}$. For $v = \fpb$, we know that $T / T_{\fpb}^+$ is unramified, with geometric Frobenius acting as $\psi(\fpb)^{-1} \alpha^{-1}$; by assumption this quantity is not congruent to 1 modulo $p$, so the $H^0$ vanishes.

    The hypothesis that $V_v^+$ has no cyclotomic quotient follows from the assumptions that $\alpha \psi(\fp) / p \notin \mu_{p^\infty}$ (so in particular this quantity is not 1) and that $\beta \psi(\fpb)/p \ne 1$.

    Finally, our classes $\bfz_{\fn}^{f, \psi}$ for $(\fn, p) = 1$ have good reduction everywhere, by Propositions \ref{prop:h1f} and \ref{prop:h1fp}; this is where we use the assumption $\alpha \psi(\fp) / p \notin \mu_{p^\infty}$. This completes the verification of the additional hypotheses needed to apply Theorem \ref{thm:ESmachine} to the Euler system $\{\bfz_{\fn}^{f, \psi} : \fn \nmid \mathcal{N}\}$.
   \end{proof}

  \subsection{Critical Selmer groups: motivation}

   Our final result on bounding Selmer groups will be an application of Theorem \ref{thm:ESmachine2} to bound the Selmer group of $T^\vee(1)$ with even weaker local conditions at $p$.

   Before doing so, we shall briefly explain some ideas from the Iwasawa theory of $f$ over the $\Zp^2$-extension of $K$; these ideas play no role in the proofs, but serve to motivate our choice of local conditions. Recall the definitions of the regions $\Sigma^{(i)}$ in Figure \ref{fig:hecke}.

   Let us suppose that $f$ is ordinary at $p$, so $V_E(f)|_{D_p}$ has a one-dimensional unramified subrepresentation $\mathscr{F}^+ V_E(f)$ (on which geometric Frobenius acts as multiplication by the unit root of the Hecke polynomial). If $\Psi$ is a Gr\"ossencharacter with infinity-type in $\Sigma^{(1)}$, and the local $L$-factors of $M_f(\Psi)(1)$ and its dual at $\fp$ and $\fpb$ do not vanish at $s = 1$, then for $v = \fp, \fpb$, we have
   \[ H^1_f(K_v, V_{E}(f)(\Psi)(1)) = H^1(K_{v}, \mathscr{F}^+ V_E(f)(\Psi)(1)).\]

   Meanwhile, whether or not $f$ is ordinary, for $\Psi \in \Sigma^{(2)}$ we have
   \begin{align*}
    H^1_f(K_{\fp}, V_E(f)(\Psi)(1)) &= 0,\\ H^1_f(K_{\fpb}, V_E(f)(\Psi)(1)) &= H^1(K_{\fpb}, V_E(f)(\Psi)(1)),
   \end{align*}
   and similarly for $\Sigma^{(2')}$ with $\fp$ and $\fpb$ reversed. In each of these critical regions, the Bloch--Kato conjecture predicts that the Selmer group $\Sel(K, T_\cO(f)(\Psi)(1) \otimes \Qp/\Zp)$ is controlled by the algebraic part of the critical $L$-value $L(f/K, \Psi, 1)$; in particular, for a ``generic'' character in these regions the Bloch--Kato Selmer group should be finite.

   Passing to a direct limit over extensions of $K$ contained in $K(\ff p^\infty)$, we obtain three Selmer groups $\Sel(K(\ff p^\infty), T_\cO(f)(1) \otimes \Qp/\Zp, \Sigma^{(i)})$, which are $\Lambda(H_{\ff p^\infty})$-modules interpolating the Bloch--Kato Selmer groups for critical $\Psi$'s in the corresponding regions. These are the algebraic counterparts of the three $p$-adic $L$-functions defined in the previous section.

   The theorem of the next subsection should then be understood as follows. We shall show, roughly, that if we specialize either of the groups $\Sel(K(\ff p^\infty), T_\cO(f)(1) \otimes \Qp/\Zp, \Sigma^{(1)})$ and $\Sel(K(\ff p^\infty), T_\cO(f)(1) \otimes \Qp/\Zp, \Sigma^{(2)})$ at a character $\psi_{\fP}$ of $H_{\ff p^\infty}$ corresponding to a Gr\"ossen\-character of conductor prime to $p$ and infinity-type $(-1, 0)$ -- thus lying in $\Sigma^{(3)}$, rather than any of the three critical regions -- then this specialization is controlled by the value at $\psi$ of the corresponding $p$-adic $L$-function.

  \subsection{Critical Selmer groups: the theorems}
   \label{sect:criticalSel}

   Let $T = T_{\cO}(f)^*(\psi^{-1})$, as before, so that
   \[ T^\vee(1) = \left(\frac{V_E(f)}{T_{\cO}(f)}\right)(\psi)(1).\]
   Throughout this section we continue to impose the assumptions \ref{assump:ord}. We shall define two Selmer groups $\Sel(K, T^\vee(1), \Sigma^{(1)})$ and $\Sel(K, T^\vee(1), \Sigma^{(2)})$.

   \begin{definition}
    \begin{enumerate}[(i)]
     \item The group $\Sel(K, T^\vee(1), \Sigma^{(1)})$ consists of all classes $c \in \Sel^{\Sigma_p}(K, T^\vee(1))$ such that for $v = \fp, \fpb$ we have
     \[ \loc_v(c) \in \operatorname{image} H^1\left(K_v, \mathscr{F}^+ V_E(f)(\psi)(1)\right).\]
     \item The group $\Sel(K, T^\vee(1), \Sigma^{(2)})$ consists of all classes $c \in \Sel^{\Sigma_p}(K, T^\vee(1))$ such that $\loc_{\fp}(c) = 0$ (with no condition on $\loc_{\fpb}(c)$).
    \end{enumerate}
   \end{definition}

   Note that the Bloch--Kato Selmer group $\Sel(K, T^\vee(1))$ is exactly the intersection of the groups $\Sel(K, T^\vee(1), \Sigma^{(1)})$ and $\Sel(K, T^\vee(1), \Sigma^{(2)})$.

   We now relate these Selmer groups to linear functionals on the local $H^1_f$ defined using the Bloch--Kato logarithm map. Recall the vectors $\eta_f^{\ur} \otimes \omega_{g}$ and $\omega_{g}^{\ur} \otimes \omega_{f}$ appearing in \S 6 above. In our situation, for $g = g_{\psi}$ a CM form and $p$ split, we have $V_E(g_\psi) \cong \Ind_{K}^\QQ(\psi_{\fP})$; and thus
   \[ \Fil^1 \DdR(V_E(f) \otimes V_E(g)) = \Fil^1 \DdR(K_\fp, V^*) \oplus \Fil^1\DdR(K_{\fpb}, V^*),\]
   since $V^* = V_E(f)(\psi_{\fP})$. Clearly we have
   \[\eta_f^{\ur} \otimes \omega_g \in \Fil^1 \DdR(K_\fp, V^*) = \DdR(K_\fp, V^*),\] and $\omega_f \otimes \eta_g^{\ur} \in \Fil^1 \DdR(K_{\fpb}, V^*)$.

   \begin{proposition}
    \begin{enumerate}[(i)]
     \item The kernel of the linear functional $\lambda_1: H^1_f(K_{\fp}, V^*) \to E$ given by
     \[ x \mapsto \left\langle \log_{K_{\fp}, V}(x), \eta_f^{\ur} \otimes \omega_g\right\rangle\]
     is $H^1\left(K_\fp, \left(\frac{V_E(f)}{\mathscr{F}^+ V_E(f)}\right)^*(\psi^{-1})\right)$.
     \item The linear functional $\lambda_2: H^1_f(K_{\fpb}, V^*) \to E$ given by
    \[ x \mapsto \left\langle \log_{K_{\fpb}, V}(x), \omega_f \otimes \eta_g^{\ur}\right\rangle\]
    is injective.
    \end{enumerate}
    Equivalently, for $i = 1, 2$, the local condition defining $\Sel(K, T^\vee(1), \Sigma^{(i)})$ is the orthogonal complement of the kernel of $\lambda_i$.
   \end{proposition}

   \begin{proof}
    Our local assumptions at $\fp$ imply that the Bloch--Kato logarithm is an isomorphism of $E$-vector spaces from $H^1(K_\fp, V) = H^1_f(K_\fp, V)$ to $\DdR(K_{\fp}, V)$. The orthogonal complement of $\eta_f^{\ur} \otimes \omega_g$ is the eigenspace of slope 2, which corresponds to $\DdR$ of the subrepresentation $\left(\frac{V_E(f)}{\mathscr{F}^+ V_E(f)}\right)^*(\psi^{-1})$. Thus the kernel of $\lambda_1$ is exactly the cohomology of this subrepresentation.

    Likewise, our local assumptions at $\fpb$ imply that of 1-dimensional $E$-vector spaces
    \[ H^1_f(K_{\fpb}, V) \rTo^\cong \left( \Fil^1 \DdR(K_{\fpb}, V^*)\right)^*,\] and $\omega_f \otimes \eta_g^{\ur}$ is a nonzero element of $\Fil^1 \DdR(K_{\fpb}, V^*)$, so the linear functional $\lambda_1$ given by pairing with this element is injective.
   \end{proof}

   Applying Theorem \ref{thm:ESmachine2} gives the following:

   \begin{corollary}
    Let $i \in \{1, 2\}$. If $\lambda_i\left(\loc_p \bfz^{f, \psi}\right) \ne 0$, then the Selmer group $\Sel(K, T^\vee(1), \Sigma^{(i)})$ is finite.

    If, in addition, $f$ has big image at $\fP$, then we have
    \begin{equation}
     \label{eq:bound}
     \ell_{\cO}\left(\Sel(K, T^\vee(1), \Sigma^{(i)})\right) \le v_{\fP} \lambda_i\left(\loc_p \bfz^{f, \psi}\right)+  c_i,
    \end{equation}
    where $c_i$ is the integer such that $\lambda_i \left(H^1_f(K \otimes \Qp, T)\right) = \fP^{-\nu_i}\cO$.
   \end{corollary}

   We now relate the right-hand side to an $L$-value. By Proposition 6.10.8 of \cite{leiloefflerzerbes13}, the quantities $c_1$ and $c_2$ are bounded above in terms of the congruence ideals $I_f$ and $I_{g_\psi}$ of $f$ and $g_{\psi}$ respectively (cf.~\cite[Definition 6.10.4]{leiloefflerzerbes13}). On the other hand, Theorem \ref{thm:Lvalue} tells us that
   \[ \lambda_1(\bfz^{f, \psi}) = -\frac{\mathcal{E}(f) \mathcal{E}^*(f)}{\mathcal{E}(f, \psi, 1)} L_\fP(f, \Sigma^{(1)})(\psi),\]
   and similarly
   \[  \lambda_2(\bfz^{f, \psi}) = -\frac{\mathcal{E}(g_{\psi}) \mathcal{E}^*(g_{\psi})}{\mathcal{E}(\psi, f, 1)} L_\fP(f, \Sigma^{(2)})(\psi).\]
    Substituting and deleting factors which are obviously in $\cO^\times$, we obtain:

   \begin{theorem}
    \label{thm:finitecriticalSel}
    If $f$ has big image at $\fP$, we have the bounds
    \[ \ell_{\cO}\left(\Sel(K, T^\vee(1), \Sigma^{(1)})\right) \le v_{\fP}\left(\frac{(1 - p^{-1} \beta \alpha^{-1})}{(1 - p^{-1} \beta \psi(\fpb))} L_\fP(f, \Sigma^{(1)})(\psi)\right) + v_{\fP}(I_f)\]
    and
    \[ \ell_{\cO}\left(\Sel(K, T^\vee(1), \Sigma^{(2)})\right) \le v_{\fP}\left(\frac{(1 - p^{-1} \psi(\fp) \psi(\fpb)^{-1})}{(1 - p^{-1} \psi(\fp) \alpha)} L_\fP(f, \Sigma^{(2)})(\psi)\right) + v_{\fP}(I_{g_\psi}).\]
   \end{theorem}

   \begin{remark}
    The appearance of the factors $v_{\fP}(I_f)$ and $v_{\fP}(I_{g_\psi})$ is a consequence of our normalization of periods: the $L$-functions $L_\fP(f, \Sigma^{(i)})$ are defined by interpolating the quotient of $L$-values $L(f/K, \psi, 1)$ for $\psi \in \Sigma^{(i)}$ by Petersson norms (of $f$ for $i = 1$, and of the appropriate CM form $g_{\lambda}$ for $i = 2$). The congruence ideals $I_f$ and $I_{g_{\psi}}$ are related to the quotients $\frac{\langle f, f \rangle}{\Omega^+_f \Omega^-_f}$, where $\Omega^\pm_f$ are the canonical periods, and similarly for $g_{\psi}$; these are essentially the algebraic parts of critical values of the adjoint $L$-function.
   \end{remark}

\appendix

 \section{Proofs of the norm relations}
  \label{sect:appendix}
  In this appendix, we give the proof of Theorem \ref{thm:asymmetricnorm}.

  \subsection{Preliminaries}

   Recall the definition of the modular curve $Y(m, N)$, for integers $m \ge 1$ and $N \ge 5$ with $m \mid N$, given in \cite[\S 2.1]{leiloefflerzerbes13}. The curve $Y(m, N)$ is an irreducible variety over $\QQ$, but it is not geometrically connected if $m \ge 3$, since there is a surjective map $Y(m, N) \to \Spec \QQ(\mu_m)$ with geometrically connected fibres (Definition 2.1.6 of \emph{op.cit.}). When we take products such as $Y(m, N)^2$ or $Y(m, N) \times Y(m, N')$, we shall always understand the fibre product to be over $\Spec \QQ(\mu_m)$ (not over $\Spec \QQ$).

   For $m, N$ as above, $c > 1$ an integer coprime to $6N$, and $j \in \ZZ / m \ZZ$, let ${}_c \cZ(m, N, j)$ denote the class in $\CH^2(Y(m, N)^2, 1)$ constructed in \S 2.6 of \emph{op.cit.}. (We have made a slight change of notation from \emph{op.cit.}; in the notation of our previous work this class would be denoted by ${}_c \cZ_{m, N/m, j}$.)

   Given integers $N, N' \ge 5$, both divisible by $m$, we define
   \[ {}_c \cZ(m, N, N', j) \in \CH^2(Y(m, N) \times Y(m, N'), 1)\]
   as the pushforward of ${}_c \cZ(m, R, j)$ along the natural degeneracy map $Y(m, R)^2 \to Y(m, N) \times Y(m, N')$, for some integer $R$ divisible by $N$ and $N'$ and with the same prime factors as $NN'$. As in \S \ref{sect:asymmzetadefs} above, this element is independent of the choice of $R$, by Theorem 3.1.1 of \cite{leiloefflerzerbes13}.

   For $\ell$ prime, we write $\pr_1, \pr_2$ for the maps $Y(m, N\ell) \to Y(m, N)$ given by $z \mapsto z$ and $z \mapsto \ell z$, as in the $Y_1$ case above.

  \subsection{Norm relations for symmetric $\cZ$'s}

   \begin{lemma}
    \label{lemma:asymmetricnorm}
    We have
     \[
      (\pr_1 \times \pr_2)_* \left( {}_c \cZ(m, \ell N, j) \right) =
      \begin{cases}
       (U_\ell', 1) \cdot {}_c \cZ(m, N, \ell j) & \text{if $\ell \mid N$,}\\
       \left[(T_\ell', 1) \Delta_{\ell^{-1}} - (\langle \ell^{-1} \rangle, T_\ell') \Delta_{\ell^{-2}} \right] \cdot {}_c \cZ(m, N, j) & \text{if $\ell \nmid N$,}
      \end{cases}
     \]
     where $\Delta_x$, for $x \in (\ZZ / m\ZZ)^\times$, denotes the action of any element of $\GL_2(\ZZ / N \ZZ)^2$ of the form $\left(\begin{pmatrix} y & 0 \\ 0 & 1\end{pmatrix}, \begin{pmatrix} y & 0 \\ 0 & 1\end{pmatrix}\right)$ with $y = x \bmod m$, and in the second case $\langle \ell^{-1} \rangle$ denotes the action of the element $\begin{pmatrix} \ell & 0 \\ 0 & \ell^{-1}\end{pmatrix} \in \SL_2(\ZZ / N\ZZ)$.
   \end{lemma}

   \begin{proof}
    Consider the intermediate modular curve $Y(m, N(\ell))$ (notation as in \cite[\S 2.8]{kato04}). Both $\pr_1$ and $\pr_2$ factor through the natural projection $\alpha: Y(m, N\ell) \to Y(m, N(\ell))$, and we have a commutative diagram
    \begin{diagram}
     Y(m, N\ell) & \rInto^{\left(1, \stbt 1 j 0 1\right)} & Y(m, N\ell)^2\\
     \dTo^\alpha & & \dTo_{\alpha \times \alpha} \\
     Y(m, N(\ell)) & \rInto^{\left(1, \stbt 1 j 0 1\right)} & Y(m, N(\ell))^2.
    \end{diagram}
    Let $\cC_{m, N(\ell), j}$ be the image of the lower horizontal map. The pushforward of ${}_c g_{0, 1/N\ell} \in \cO(Y(m, N\ell))^\times$ to $\cO(Y(m, N(\ell)))^\times$ is given by
    \[ \varphi_\ell^* \left( {}_c g_{0, 1/N}\right)\]
    if $\ell \mid N$, and by
    \[ \varphi_\ell^*\left( {}_c g_{0, 1/N}\right) \cdot \left({}_c g_{0, \quot{\ell^{-1}}/N}\right)^{-1}\]
    if $\ell \nmid N$; see \cite[\S 2.13]{kato04}. Here $\varphi_\ell$ is the map $Y(m, N(\ell)) \to Y(m(\ell), N)$ given by $z \mapsto \ell z$. Thus
    \[
     (\alpha \times \alpha)_*\left({}_c \cZ(m, N\ell, j)\right) =
     \begin{cases}
      \left(\cC_{m, N(\ell), j}, \varphi_\ell^* \left({}_c g_{0, 1/N}\right)\right) & \text{if $\ell \mid N$,}\\
      \left(\cC_{m, N(\ell), j}, \varphi_\ell^* \left({}_c g_{0, 1/N}\right)\right) - \left(\cC_{m, N(\ell), j}, {}_c g_{0, \quot{\ell^{-1}}/N}\right) & \text{if $\ell \nmid N$.}
     \end{cases}
    \]

    Now let $\pi_1$ and $\pi_2$ be the degeneracy maps $Y(m, N(\ell)) \to Y(m, N)$, so that $\pr_i = \pi_i \circ \alpha$. We must study the image of the elements given above under pushforward by the map $\pi_1 \times \pi_2$.
    We claim that:
    \begin{itemize}
     \item If $\ell \mid N$, then
     \begin{equation}
      \label{eq:normrel0}
      (\pi_1 \times \pi_2)_* \left(\cC_{m, N(\ell), j}, \varphi_\ell^* \left({}_c g_{0, 1/N}\right)\right) = (U_\ell', 1) \cdot {}_c \cZ(m, N, \ell j).
     \end{equation}
     \item If $\ell \nmid N$, then
     \begin{subequations}
      \begin{equation}
       \label{eq:normrel1}
       (\pi_1 \times \pi_2)_* \left(\cC_{m, N(\ell), j}, \varphi_\ell^* \left({}_c g_{0, 1/N}\right)\right) = (T_\ell', 1) \cdot {}_c \cZ(m, N, \ell j)
      \end{equation}
      and
      \begin{equation}
       \label{eq:normrel2}
       (\pi_1 \times \pi_2)_* \left(\cC_{m, N(\ell), j}, {}_c g_{0, \quot{\ell^{-1}}/N}\right) = (\langle \ell^{-1} \rangle, T_\ell') \sigma_\ell^{-2} \cdot {}_c \cZ(m, N, j).
      \end{equation}
     \end{subequations}
    \end{itemize}

    For formulae \eqref{eq:normrel0} and \eqref{eq:normrel1}, we use the isomorphism $\varphi_\ell: Y(m, N(\ell)) \cong Y(m(\ell), N)$ to write
    \[ (\pi_1 \times \pi_2)_* \left(\cC_{m, N(\ell), j}, \varphi_\ell^* \left({}_c g_{0, 1/N}\right)\right) = (\pi_2' \times \pi_1')_* \left(\cC_{m, N(\ell), j}', {}_c g_{0, 1/N}\right)\]
    where $\cC_{m, N, j}'$ is the locus of points in $Y(m(\ell), N)$ of the form $(z, z + \ell j)$, and $\pi_1', \pi_2': Y(m(\ell), N) \to Y(m, N)$ are given by $z \mapsto z$ and $z \mapsto z/\ell$ respectively. However, one sees readily that under $1 \times \pi_1'$, $\cC_{m, N(\ell), j}'$ maps isomorphically to its image in $Y(m(\ell), N) \times Y(m, N)$, and this image coincides with the inverse image of $\cC_{m, N, \ell j}$ under the map $\pi_1' \times 1$. Hence
    \[  (\pi_2' \times \pi_1')_* \left(\cC_{m, N(\ell), j}', \left({}_c g_{0, 1/N}\right)\right) = (\pi_2' \times 1)_* (\pi_1' \times 1)^* \left(\cC_{m, N, \ell j}, {}_c g_{0, 1/N}\right);\]
    and the map $(\pi_2' \times 1)_* (\pi_1' \times 1)^*$ is the definition of the operator $(U_\ell', 1)$ or $(T_\ell', 1)$ in the cases $\ell \mid N$ or $\ell \nmid N$ respectively.

    For formula \eqref{eq:normrel2}, we note similarly that $\cC_{m, N, j}$ maps isomorphically to its image $Y(m, N) \times Y(m, N(\ell))$; and if we temporarily write ${}_c \cZ(m, N, j, \alpha)$, for $\alpha \in \ZZ / N\ZZ$, for the analogue of ${}_c \cZ(m, N, j)$ formed with ${}_c g_{0, \alpha/N}$ in place of ${}_c g_{0, 1/N}$, then it is immediate from the definitions that
    \[ (\pi_1 \times \pi_2)_* \left(\cC_{m, N(\ell), j}, {}_c g_{0, \quot{\ell^{-1}}/N}\right) = (1, T_\ell) \cdot {}_c \cZ(m, N, j, \ell^{-1}),\]
    (since $T_\ell$ acts as $(\pi_2)_* (\pi_1)^*$). But we also have the relation
    \[ {}_c \cZ(m, N, j, \ell^{-1}) = \left(\begin{pmatrix} \ell^{-1} & 0 \\ 0 & \ell^{-1} \end{pmatrix}, \begin{pmatrix} \ell^{-1} & 0 \\ 0 & \ell^{-1} \end{pmatrix}\right) \cdot {}_c \cZ(m, N, j) = (\langle \ell^{-1} \rangle, \langle \ell^{-1} \rangle) \Delta_{\ell^{-2}} \cdot {}_c\cZ(m, N, j),\]
    and $T_\ell = \langle \ell \rangle T_\ell'$ (see \cite[\S 4.9]{kato04}), hence
    \begin{align*}
      (\pi_1 \times \pi_2)_* \left(\cC_{m, N(\ell), j}, {}_c g_{0, \quot{\ell^{-1}}/N}\right) &= (1, T_\ell) \cdot {}_c \cZ(m, N, j, \ell^{-1})\\
      &= (1, T_\ell) \cdot \left(\langle \ell^{-1} \rangle, \langle \ell^{-1} \rangle\right) \Delta_\ell^{-2} \cdot {}_c \cZ(m, N, j)\\
      &= (\langle \ell^{-1} \rangle, T_\ell') \Delta_{\ell^{-2}} \cdot {}_c \cZ(m, N, j)
     \end{align*}
    as required.
   \end{proof}

   \begin{remark}
    As we shall see in the following subsections, all of the norm relations we use in both this paper and our previous paper \cite{leiloefflerzerbes13} can be derived from Theorems 3.1.1 and 3.3.1 of \cite{leiloefflerzerbes13} and the above lemma, using only elementary identities for Hecke operators and pushforward maps.
   \end{remark}

  \subsection{Norm relations for asymmetric $\cZ$'s}

   We now state and prove a theorem which is the analogue of Theorem \ref{thm:asymmetricnorm} for the elements ${}_c \cZ(m, N, N', j)$.

   \begin{theorem}
    \label{thm:asymmetricnormz}
    Let $m \ge 1, N, N' \ge 5$ be integers with $m \mid N$ and $m \mid N'$, $\ell$ a prime, $j \in \ZZ / m\ZZ$, and $c > 1$ an integer coprime to $6\ell NN'$.
    \begin{enumerate}[(a)]
     \item We have
     \[ (1 \times \pr_1)_* \left( {}_c \cZ(m, N, \ell N', j) \right) =
      \begin{cases}
       {}_c \cZ(m, N, N', j)  & \text{if $\ell \mid N N'$,}\\
       \left[1 - \left(\stbt {\ell^{-1}} 00{\ell^{-1}}, \stbt {\ell^{-1}} 00{\ell^{-1}}\right)^*\right] \cdot \cZ(m, N, N', j) & \text{if $\ell \nmid N N'$,}
      \end{cases}
     \]
     where in the latter case $\left(\stbt {\ell^{-1}} 00{\ell^{-1}}, \stbt {\ell^{-1}} 00{\ell^{-1}}\right)$ is considered as an element of $\GL_2(\ZZ / N \ZZ) \times \GL_2(\ZZ / N'\ZZ)$.
     \item We have
     \begin{multline*}
      (1 \times \pr_2)_* \left( {}_c \cZ(m, N, \ell N', j) \right)
      \\=
      \begin{cases}
       (U_\ell', 1) \cdot {}_c \cZ(m, N, N', \ell j) & \text{if $\ell \mid N$,}\\
       \left[(T_\ell', 1) \sigma_\ell^{-1} - (\langle \ell^{-1} \rangle, U_\ell') \sigma_\ell^{-2} \right] \cdot {}_c \cZ(m, N, N', j) & \text{if $\ell \nmid N$ but $\ell \mid N'$,}\\
       \left[(T_\ell', 1) \sigma_\ell^{-1} - (\langle \ell^{-1} \rangle, T_\ell') \sigma_\ell^{-2} \right] \cdot {}_c \cZ(m, N, N', j) & \text{if $\ell \nmid NN'$.}\\
      \end{cases}
     \end{multline*}
     where in the second and third cases $\sigma_\ell^{-1}$ denotes any element of $\GL_2(\ZZ / N \ZZ) \times \GL_2(\ZZ / N'\ZZ)$ congruent to $\left(\stbt {\ell^{-1}} 001, \stbt {\ell^{-1}} 001\right)$ modulo $m$.
    \end{enumerate}
   \end{theorem}

   \begin{proof}
    Part (i) is immediate from Theorem 3.1.1 of \cite{leiloefflerzerbes13} (we have only included it here for completeness).

    We will reduce part (ii) to properties of the ``symmetric'' zeta elements ${}_c \cZ(m, N, j)$. As usual, let $R$ be an integer divisible by $N$ and $N'$ and with the same prime factors as $NN'$. We have a commutative diagram
    \begin{diagram}
     Y(m, \ell R)^2 & \rTo & Y(m, N) \times Y(m, \ell N') \\
     \dTo^{\pr_1 \times \pr_2} & & \dTo_{1 \times \pr_2} \\
     Y(m, R)^2 & \rTo & Y(m, N) \times Y(m, N')
    \end{diagram}
    where the horizontal arrows are the natural degeneracy maps; and the elements ${}_c \cZ(m, N, \ell N', j)$ and ${}_c \cZ(m, N, N', j)$ are by definition the pushforwards of ${}_c \cZ(m, \ell R, j)$ and ${}_c \cZ(m, R, j)$ along these horizontal maps.

    If $\ell \mid N$, then $\ell \mid R$, so we may apply the first case of Lemma \ref{lemma:asymmetricnorm} to deduce that
    \[ (\pi_1 \times \pi_2)_* {}_c \cZ(m, \ell R, j) = (U_\ell', 1) {}_c \cZ(m, R, \ell j).\]
    The assumption that $\ell \mid N$ implies that $U_\ell'$ commutes with the pushforward map $Y(m, R) \to Y(m, N)$, so we are done in this case.

    If $\ell \nmid N$, but $\ell \mid N'$, then the pushforward from $Y(m, R)$ to $Y(m, \ell N)$ commutes with $U_\ell'$, but from level $N\ell$ to level $N$ we have the commutation relation $(\pr_1)_* \circ U_\ell' = T_\ell' \circ (\pr_1)_* - \langle \ell^{-1} \rangle \circ (\pr_2)_*$. Thus the pushforward of $(U_\ell', 1) {}_c \cZ(m, R, \ell j)$ is
    \[ (T_\ell', 1) {}_c \cZ(m, N, \ell j) - (\langle \ell^{-1} \rangle, 1) (\pr_2 \times \pr_1)_* {}_c \cZ(m, \ell N, N', \ell j).\]
    Since $\ell \mid N'$ we can apply the previously-considered case to conclude that
    \[ (\pr_2 \times \pr_1)_* {}_c \cZ(m, \ell N, N', \ell j) = (1, U_\ell') {}_c \cZ(m, N, N', \ell^2 j)\]
    as required.

    This leaves only the case $\ell \nmid NN'$. Then $\ell \nmid R$, so the pushforward $Y(m, R)^2 \to Y(m, N) \times Y(m, N')$ commutes with $(T_\ell', 1)$ and $(1, T_\ell')$; and we are done by the second case of Lemma \ref{lemma:asymmetricnorm}.
   \end{proof}

  \subsection{Norm relations for ${}_c \Xi$'s: proof of Theorem \ref{thm:asymmetricnorm}}

   We now deduce Theorem \ref{thm:asymmetricnorm} from Theorem \ref{thm:asymmetricnormz}. Let us begin by recalling the relation between the classes $\cZ(m, N, N', j)$ of the preceding sections and the classes ${}_c \Xi(m, N, N', j)$ of Definition \ref{def:xi}.

   Recall the map $t_m: Y(m, mN) \to Y_1(N) \times \Spec \QQ(\mu_m)$ defined in \S 2.1 of \cite{leiloefflerzerbes13}. This map commutes with the operators $T_\ell'$ for $\ell \nmid mN$, $U_\ell'$ for $\ell \mid N$, and $\langle d \rangle$ for all $d$. Moreover, it intertwines the action of $\begin{pmatrix} \ell & 0 \\ 0 & 1 \end{pmatrix}$ with the arithmetic Frobenius $\sigma_\ell$. Moreover, for $i = 1, 2$ we have $\pr_i \circ \mathop{t}_m = t_m \circ \pr_i$ as maps $Y(m, \ell m N) \to Y_1(N) \times \Spec \QQ(\mu_m)$.

   It is immediate from the definitions that we have
   \begin{equation}
    \label{eq:zsandxis}
    {}_c \Xi(m, N, N', j) = (t_m \times t_m)_* \left({}_c \cZ(m, mN, mN', j)\right).
   \end{equation}

   Let us now recall the statement of the theorem.

   \begin{theorem}[Theorem \ref{thm:asymmetricnorm}]
    Let $m \ge 1, N, N' \ge 5$ be integers, $\ell$ a prime, $j \in \ZZ / m\ZZ$, and $c > 1$ an integer coprime to $6 \ell mNN'$. Let $\pr_1, \pr_2$ be the two degeneracy maps $Y_1(\ell N') \to Y_1(N')$, corresponding to $z \mapsto z$ and $z \mapsto \ell z$ respectively.
    \begin{enumerate}[(a)]
     \item We have
     \[ (1 \times \pr_1)_* \left( {}_c \Xi(m, N, \ell N', j) \right) =
      \begin{cases}
       {}_c \Xi(m, N, N', j)  & \text{if $\ell \mid m N N'$,}\\
       \left[1 - (\langle \ell^{-1}\rangle, \langle \ell^{-1} \rangle) \sigma_\ell^{-2}\right] \cdot {}_c \Xi(m, N, N', j) & \text{if $\ell \nmid m N N'$.}
      \end{cases}
     \]
     \item The pushforward $(1 \times \pr_2)_* \left( {}_c \Xi(m, N, \ell N', j) \right)$ is given by the following formulae:
     \begin{enumerate}[(i)]
      \item if $\ell \mid N$, then
      \[ (1 \times \pr_2)_* \left( {}_c \Xi(m, N, \ell N', j) \right) = (U_\ell', 1) \cdot {}_c \Xi(m, N, N', \ell j);\]
      \item if $\ell \nmid N$ but $\ell \mid N'$,then
      \[ (1 \times \pr_2)_* \left( {}_c \Xi(m, N, \ell N', j) \right) = (T_\ell', 1) \cdot {}_c \Xi(m, N, N', \ell j) - (\langle \ell^{-1} \rangle, U_\ell') \cdot {}_c \Xi(m, N, N', \ell^2 j);\]
      \item if $\ell \nmid mNN'$, then
      \[ (1 \times \pr_2)_* \left( {}_c \Xi(m, N, \ell N', j) \right) =
       \left[(T_\ell', 1) \sigma_\ell^{-1} - (\langle \ell^{-1} \rangle, T_\ell') \sigma_\ell^{-2} \right] \cdot {}_c \Xi(m, N, N', j).
      \]
     \end{enumerate}
    \end{enumerate}
   \end{theorem}

  \begin{proof}

   Using equation \eqref{eq:zsandxis}, part (a) of the theorem follows directly from Theorem \ref{thm:asymmetricnormz}(a), and many cases of part (b) follow from Theorem \ref{thm:asymmetricnormz}(b): more precisely, all the cases where $\ell \nmid m$ are immediate, as are all the cases where $\ell \mid N$, since in these cases the map $(t_m \times t_m)_*$ intertwines the relevant Hecke operators on $Y(m, mN) \times Y(m, mN')$ with those on $Y_1(N) \times Y_1(N') \times \QQ(\mu_m)$.

   The only case that remains is (ii) with $\ell \mid N$. In this case, we can argue that
   \begin{align*}
    (1 \times \pr_2)_* {}_c \Xi(m, N, \ell N', j) &= (1 \times \pr_2)_* (\pr_1 \times 1)_* {}_c \Xi(m, \ell N, \ell N', j) \\
    &= (\pr_1 \times 1)_* (1 \times \pr_2)_* {}_c \Xi(m, \ell N, \ell N', j) \\
    &= (\pr_1 \times 1)_* (U_\ell', 1) {}_c \Xi(m, \ell N, N', \ell j) \\
    &= (T_\ell', 1) (\pr_1 \times 1)_* {}_c \Xi(m, \ell N, N', \ell j) \\ &\qquad - (\langle \ell^{-1}\rangle, 1) (\pr_2 \times 1)_* {}_c \Xi(m, \ell N, N', \ell j).
   \end{align*}
   Since $\ell \mid N'$, both of these terms can be calculated using previously-considered cases of the present theorem: the first term is $(T_\ell', 1) {}_c \Xi(m, N, N', \ell j)$, by part (a), while the second term is $(\langle \ell^{-1}\rangle, U_\ell') {}_c \Xi(m, N, N', \ell^2 j)$ by part (b)(i) (with the roles of $N$ and $N'$ interchanged).
  \end{proof}

  \begin{remark}
   \label{remark:higherpowers}
   In the above theorem, we excluded the most awkward case, which is when $\ell \mid m$ but $\ell \nmid NN'$. We briefly indicate how to obtain a formula in this case as well. In this setting, applying the argument of the final paragraph of the proof above shows that
   \[
    (1 \times \pr_2)_* {}_c \Xi(m, N, \ell N', j) = (T_\ell', 1){}_c \Xi(m, N, N',\ell j) \\- (\langle \ell^{-1}\rangle, 1) (\pr_2 \times 1)_* {}_c \Xi(m, \ell N, N', \ell j).
   \]
   Proceeding inductively, interchanging the roles of $N$ and $N'$ at each step, we find that for any $h \ge 0$ we have
   \begin{multline*}
    (1 \times \pr_2)_* {}_c \Xi(m, N, \ell N', j) = \\
    (T_\ell', 1) \sum_{\substack{1 \le a \le h \\ \text{$a$ odd}}} (\langle \ell^{-1}\rangle, \langle \ell^{-1}\rangle)^{(a-1)/2} {}_c \Xi(m, N, N', \ell^a j)\\
    - (\langle \ell^{-1} \rangle, T_\ell') \sum_{\substack{2 \le a \le h \\ \text{$a$ even}}} (\langle \ell^{-1}\rangle, \langle \ell^{-1}\rangle)^{(a-2)/2} {}_c \Xi(m, N, N', \ell^a j) \\ +
    \begin{cases}
     (\langle \ell^{-1}\rangle^{h/2}, \langle \ell^{-1}\rangle^{h/2}) (1 \times \pr_2)_* {}_c \Xi(m, N, \ell N',\ell^{h} j) & \text{if $h$ even,}\\
     - (\langle \ell^{-1}\rangle^{(h+1)/2}, \langle \ell^{-1}\rangle^{(h-1)/2}) (\pr_2 \times 1)_* {}_c \Xi(m, \ell N, N',\ell^{h} j)& \text{if $h$ odd.}
    \end{cases}
   \end{multline*}
   If we take $h = v_p(m)$, then ${}_c \Xi(m, \ell N, N',\ell^h j) = {}_c \Xi(\ell^{-h} m, \ell N, N', j)$ etc, and we can now apply the formulae in the $\ell \nmid mNN'$ case previously studied.
  \end{remark}

 \section{Euler systems with crystalline local conditions}
  \label{sect:appendixB}

  In this appendix we'll prove some theorems which are slight variations on the results of \cite{rubin00}. This section is the outcome of an email exchange with Karl Rubin and we are very grateful to him for his patient explanations; any mistakes below are, however, ours.

  \subsection{Local properties of Kolyvagin classes}

   Let $K$ be a number field, $\fn$ an integral ideal of $K$, and $\mathcal{K}$ a pro-$p$ extension\footnote{This is perhaps not quite standard terminology: we mean that $\mathcal{K}$ is a possibly infinite extension of $K$ which is a union of finite extensions of $p$-power degree.} of $K$ containing $K(\fq)$ for every prime $\fq \nmid \fn$. We consider a finite extension $E/\Qp$ with ring of integers $\cO$ and residue field $\mathbf{k}$, and a finite-rank free $\cO$-module $T$ with an action of $\Gal(\overline{K} /K)$ unramified outside the primes dividing $\fn$. For $M \in \cO$, let $W_M = T / M T$.

   Let $\bfc = \{\bfc_F : K \subset_f F \subset \mathcal{K}\}$ an Euler system for $(T, \mathcal{K}, \fn)$ in the sense of \cite{rubin00}. Recall the construction -- cf.~\cite[\S 4.4]{rubin00} -- of ``Kolyvagin derivative'' classes
   \[ \kappa_{[\fr, M]} \in H^1(K, W_M) \]
   for each $\fr \in \mathcal{R}_{M}$, where $\mathcal{R}_{M} = \mathcal{R}_{K, M}$ is the set of ideals of $K$ defined in Definition 4.1.1 of \emph{op.cit.}.

   We shall not need the details of the construction here; let it suffice to note the following property:

   \begin{proposition}[{cf.~\cite[Proposition 4.4.13]{rubin00}}]
    The restriction
    \[ \operatorname{res}_{K(\fr) / K} \left(\kappa_{[\fr, M]}\right) \in H^1(K(\fr), W_M)\]
    is the image modulo $M$ of $D_{\fr} \left(\bfc_{K(\fr)}\right) \in H^1(K(\fr), T)$, where $D_{\fr}$ is a certain element of the group ring $\ZZ[\Gal(K(\fr) / K)]$.
   \end{proposition}

   We are interested in the local properties of $\kappa_{[\fr, M]}$ at primes of $K$ not dividing $\fr$ (but possibly dividing $p$). Let $v$ be a prime of $K$ dividing $\fn$. We make the following assumption:

   \begin{assumption}
    \label{assump:h1f}
    The following conditions are satisfied:
    \begin{enumerate}[(i)]
     \item There exists a subspace $V^+ \subseteq V$ stable under $G_{K_v}$.
     \item We have
     \[ H^0(K_v, (T / T^+) \otimes \mathbf{k}) = 0,\]
     where $T^+ = T \cap V^+$.
    \end{enumerate}
   \end{assumption}

   \begin{remark}
    Note that if assumption (ii) is satisfied, we automatically have the apparently stronger result that $H^0(L, (T / T^+) \otimes \mathbf{k}) = 0$ for any finite Galois extension $L/K_v$ of $p$-power degree, since if $H^0(L, (T / T^+) \otimes \mathbf{k})$ were nonzero, it would be a finite-dimensional $\mathbf{F}_p$-vector space equipped with an action of the finite $p$-group $\Gal(L/K_v)$, so it would necessarily have non-zero invariants under this $p$-group, contradicting our assumption (ii).
   \end{remark}

   \begin{theorem}
    Suppose $T$ satisfies Assumption \ref{assump:h1f}, and the Euler system $\bfc$ has the property that for every $K \subset_f F \subset \mathcal{K}$, and each prime $w \mid v$ of $F$, we have
    \[ \loc_w\left(\bfc_{F}\right) \in H^1(F_w, V^+) \subseteq H^1(F_w, V).\]
    Then for any nonzero $M \in \cO$ and any $\fr \in \mathcal{R}_{M}$, we have
    \[ \loc_v \left(\kappa_{[\fr, M]}\right) \in H^1(K_v, W_M^+) \subset H^1(K_v, W_M),\]
    where $W_M^+$ is the image of $T^+$ in $W_M$.
   \end{theorem}

   (Note that $W_M^+ = T^+ / MT^+$, since $T^+$ is saturated in $T$.)

   \begin{proof}

    From the remark above, we know that for every $K \subset_f F \subset\mathcal{K}$, and each $w \mid v$ of $F$, we have an injection $H^1(F_w, T^+) \into H^1(F_w, T)$, and the cokernel is torsion-free, so we have $H^1(F_w, T^+) = H^1(F_w, T) \cap H^1(F_w, V^+)$. So our assumption on $\bfc_F$ implies that $\loc_w\left(\bfc_{F}\right) \in H^1(F_w, T^+)$. Moreover, $\bigoplus_{w \mid v} H^1(F_w, T^+)$ is stable under the action of $\ZZ[\Gal(F/K)]$.

    Consequently, $\loc_w \left(D_{\fr} \bfc_{K(\fr)}\right) \in H^1(K(\fr)_w, T^+)$ for each $\fr$ and each prime $w \mid v$ of $K(\fr)$; and thus
    \[ \loc_w\left[ \operatorname{res}_{K(\fr) / K} \left(\kappa_{[\fr, M]}\right) \right] \in H^1(K(\fr)_w, W_M^+)\]
    whenever $\fr \in \mathcal{R}_{M}$. This is equivalent to the statement that
    \[ \pi \left( \loc_w\left[ \operatorname{res}_{K(\fr) / K} \left(\kappa_{[\fr, M]}\right) \right]\right) = 0,\]
    where $\pi$ is the map $H^1(K(\fr)_w, W_M) \to H^1(K(\fr)_w, W_M / W_M^+)$ induced by the projection $W_M \to W_M/W_M^+$.

    Equivalently, we have
    \[ \operatorname{res}_{K(\fr)_w / K_v} \left[\pi \left( \loc_v \left(\kappa_{[\fr, M]}\right) \right]\right) = 0
    \]
    for each $w \mid v$, since $\pi$ commutes with restriction. But the kernel of the restriction map
    \[ \operatorname{res}_{K(\fr)_w / K_v}: H^1(K_v, W_M / W_M^+) \to H^1(K(\fr)_w, W_M / W_M^+)\]
    is $H^1\left(K(\fr)_w / K_v, H^0(K(\fr)_w, W_M / W_M^+)\right)$, and (again by the remark above) we know that the space $H^0(K(\fr)_w, W_M / W_M^+)$ is zero. Thus $\loc_v \left(\kappa_{[\fr, M]}\right) \in \ker(\pi)$ as required.
   \end{proof}

   \begin{corollary}
    \label{cor:kolyvagin}
    Suppose that $V$ has a subspace $V^+$ preserved by the decomposition group $D_v$ at $v$, and satisfying the following conditions:
    \begin{enumerate}[(i)]
     \item the residue characteristic of $v$ is $p$,
     \item the representation $V$ is de Rham,
     \item \label{item:panchishkin} for every embedding $K_v \into \mathbf{C}_p$, all Hodge--Tate weights of $V^+$ are $\ge 1$ and all Hodge--Tate weights of $V/V^+$ are $\le 0$,
     \item there is no nonzero quotient of $V^+$ on which $G_{K_v}$ acts via the cyclotomic character,
     \item we have $H^0(K_v, (T/T^+) \otimes \mathbf{k}) = 0$, where $T$ is a lattice in $V$ and $T^+ = T \cap V^+$.
    \end{enumerate}

    Let $\bfc$ be an Euler system for $(T, \mathcal{K}, \fn)$ and suppose that for all $K \subset_f F \subset \mathcal{K}$, and all $w \mid v$, we have $\loc_w \bfc_F \in H^1_f(F_w, T)$. Then there is a power $m$ of $p$ such that for any nonzero $M \in \cO$ and any $\fr \in \mathcal{R}_{Mm}$, we have
    \[ \loc_v \left(\kappa_{[\fr, M]}\right) \in H^1_f(K_v, W_M),\]
    where $H^1_f(K_v, W_M)$ is the image of $H^1_f(K_v, T)$ in $H^1(K_v, W_M)$.
   \end{corollary}

   \begin{remark}
    If a subrepresentation $V^+$ satisfying condition (\ref{item:panchishkin}) exists, it is unique. The existence of such a subspace is sometimes referred to as the ``Panchishkin condition''.
   \end{remark}

   \begin{proof}
    Let us first show that $H^1(K_v, T^+) = H^1_f(K_v, T)$. Both sides are saturated in $H^1(K_v, T)$; this is true by definition for $H^1_f(K_v, T)$, and for $H^1(K_v, T^+)$ it is a consequence of the vanishing of $H^0(K_v, (T/T^+) \otimes \mathbf{k})$. So it suffices to check this after inverting $p$, i.e.~to check that $H^1_f(K_v, V) = H^1(K_v, V^+)$.

    We recall the formula for the dimension of $H^1_f$ of an arbitrary crystalline Galois representation:
    \[ \dim_{E} H^1_f(K_v, V) = \dim_{E}\left( \frac{\DdR(V)}{\Fil^0 \DdR(V)}\right) + \dim_{E} H^0(K_v, V).\]

    Comparing this formula for $V$ and for $V^+$, and noting that $H^0(K_v, V/V^+) = 0$ (since we are assuming the stronger statement that the $H^0$ is trivial after tensoring with $\mathbf{k}$), we see that $H^1_f(K_v, V) = H^1_f(K_v, V^+)$, and moreover that
    \[ \dim_{E} H^1_f(K_v, V^+) = [K_v : \Qp] \dim_{E}(V^+) + \dim_E H^0(K_v, V^+).\]
    By Tate's local Euler characteristic formula, we have $H^1_f(K_v, V^+) = H^1(K_v, V^+)$ if (and only if) $H^2(K_v, V^+) = 0$; but we are assuming that $V^+$ has no cyclotomic quotient, so this $H^2$ is indeed zero and the claim follows.

    Now, by the previous theorem, for any $\fr \in \mathcal{R}_M$ we have $\kappa_{[\fr, M]} \in H^1(K_v, W_M^+)$. It is not necessarily true that $H^1(K_v, T^+) \to H^1(K_v, W_M^+)$ is necessarily surjective; there is an obstruction arising from the torsion in $H^2(K_v, T^+)$. To circumvent this, we argue as in Corollary 4.6.5 of \cite{rubin00}: one knows that if $\fr \in \mathcal{R}_{Mm}$, we have $\kappa_{[\fr, M]} = m \kappa_{[\fr, Mm]}$; and since the torsion subgroup of $H^2(K_v, T^+)$ is finite, we may choose $m$ such that the multiplication-by-$m$ map
    \[ H^2(K_v, T^+)[Mm] \to H^2(K_v, T^+)[M]\]
    is the zero, from which it follows that $\kappa_{[\fr, M]} \in H^1_f(K_v, W_M)$.
   \end{proof}

  \subsection{Applications to Selmer groups}

   We now apply the results in the previous section to deduce variants of two of the main theorems of \cite{rubin00}.

   \begin{definition}
    Let $K$ be a number field, $\fn$ an integral ideal of $K$, $T$ an $\cO$-linear representation of $\Gal(\overline{K} / K)$ unramified outside $\fn$, and $\mathcal{K}$ a pro-$p$ extension of $K$ containing $K(\fq)$ for all primes $\fq \nmid \fn$.

    We say an Euler system $\bfc$ for $(T, \mathcal{K}, \fn)$ has \emph{everywhere good reduction} if for all fields $F$ with $K \subset_f F \subset \mathcal{K}$, we have $\bfc_F \in \Sel(F, T)$.
   \end{definition}

   We make the following supplementary hypothesis which we denote by ``$\Hyp(\gamma)$'': there exists $\gamma \in \Gal(\overline{K} / K)$ such that $T^{\gamma = 1} = 0$ and $\gamma$ acts trivially on the field $K(1) K(\mu_{p^\infty}, (\cO_K)^\times)^{1/p^\infty})$. We write, as usual, $\Sigma_p$ for the set of primes dividing $p$.

   \begin{theorem}
    \label{thm:ESmachine}
    Let $\bfc$ be an Euler system for $(T, \mathcal{K}, \fn)$ with everywhere good reduction. Suppose that $\Hyp(\gamma)$ holds, and that for every prime $v \mid p$, there exists a subrepresentation $V_v^+ \subseteq V$ satisfying the hypotheses of Corollary \ref{cor:kolyvagin}.

    Then:
    \begin{enumerate}
     \item If $\Hyp(K, V)$ is satisfied and $\bfc_K \notin H^1(K, T)_{\mathrm{tors}}$, then $\Sel(K, T^\vee(1))$ is finite.
     \item If $\Hyp(K, T)$ is satisfied and $p > 2$, then we have
     \[ \ell_{\cO}(\Sel(K, T^\vee(1)) \le \ind_{\cO}(\bfc_K) + n_W + n^*_W\]
     where $n_W$ and $n^*_W$ are as in \cite{rubin00}.
    \end{enumerate}
   \end{theorem}

   \begin{proof}
    We shall argue as in the modified form of Theorems 2.2.2 and 2.2.3 of \cite{rubin00} proved in \S 9.1 of \emph{op.cit.}, where it is shown that $\Hyp(\gamma)$ and the assumption that the Euler system has good reduction outside $\Sigma_p$ may be used to dispense with the more usual assumption that $\mathcal{K}$ contains at $\Zp$-extension.

    Corollary \ref{cor:kolyvagin} shows that under our hypotheses, and at the cost of possibly increasing $M$ by a finite factor, the Kolyvagin classes $\kappa_{[\fr, M]}$ are in $\Sel^{\Sigma_\fr}(K, W_M)$ (not just in $\Sel^{\Sigma_{p\fr}}(K, W_M)$). Hence Rubin's proofs go through with $\Sigma_{p\fr}$ replaced by $\Sigma_{\fr}$ throughout, and we obtain the above theorem.
   \end{proof}

   We also have a version with modified local conditions at $p$, paralleling Rubin's Theorem 2.2.10. We continue to suppose that $\Hyp(\gamma)$ holds, and that for all primes $v \mid p$ of $K$, there exists a subrepresentation $V^+ \subseteq V$ satisfying the hypotheses of Corollary \ref{cor:kolyvagin}.

   Let us choose a nonzero $E$-linear functional $\lambda$ on the space
   \[ H^1_{f}(K \otimes \Qp, V) \coloneqq \bigoplus_{v \mid p} H^1_f(K_v, V).\]
   We write $H^1_\lambda(K \otimes \Qp, T)$ for the fractional $\cO$-ideal which is the image of $H^1_{f}(K \otimes \Qp, T)$ under $\lambda$.

   Let $\Sel_\lambda(K, T^\vee(1)) \subseteq \Sel(K, T^\vee(1))$ be the Selmer group with local conditions at $v \mid p$ given by the orthogonal complement of $\ker \lambda$.

   \begin{theorem}
    \label{thm:ESmachine2}
    Let $\bfc$ be an Euler system for $(T, \mathcal{K}, \fn)$ with everywhere good reduction.
    \begin{enumerate}
     \item If $\Hyp(K, V)$ is satisfied and $\lambda(\loc_p \bfc_K) \ne 0$, then $\Sel_\lambda(K, T^\vee(1))$ is finite.
     \item If $\Hyp(K, T)$ is satisfied and $p > 2$, then we have
     \[ \ell_\cO \left(\Sel_\lambda(K, T^\vee(1))\right) \le \ell_\cO\left(\frac{H^1_\lambda(K \otimes \Qp, T)}{\cO \lambda(\loc_p \bfc_K)}\right) + n_W + n^*_W.\]
    \end{enumerate}
   \end{theorem}

   \begin{proof}
    This follows from Theorem \ref{thm:ESmachine} via exactly the same argument as Theorem 2.2.10 of \cite{rubin00} is deduced from Theorems 2.2.2 and 2.2.3 of \emph{op.cit.}.
   \end{proof}

   \begin{remark}
    The correct context for these results is clearly that of the ``Selmer structures'' of \cite{mazurrubin04}. The results of \emph{op.cit.}~are only written up for $K = \QQ$, whereas in the present paper we are interested in $K$ a quadratic extension of $\QQ$, but the generalization is routine.

    The results of the previous section show that if $\bfc$ has everywhere good reduction and the hypotheses of Corollary \ref{cor:kolyvagin} hold, then the Kolyvagin system $\pmb{\kappa}$ derived from $\bfc$ is a Kolyvagin system for the ``Bloch--Kato Selmer structure'' $\mathcal{F}_{BK}$, where $\mathcal{F}_{BK}$ is given by the $H^1_f$ local conditions at all primes (including $v \mid p$).

    In the theory of \cite{mazurrubin04} a major role is played by a quantity $\chi(T) = \chi(T, \mathcal{F})$ attached to the representation $T$ and the Selmer structure $\mathcal{F}$ (cf.~Definition 5.2.4 of \emph{op.cit.}). The module of Kolyvagin systems is zero if $\chi(T) = 0$, free of rank one over $\cO$ if $\chi(T) = 1$, and not even finitely-generated over $\cO$ if $\chi(T) > 1$.

    If we define $\mathcal{F}_B$ to be the Selmer structure given by the canonical $H^1_f$ local condition at primes away from $p$, and at $p$ by some arbitrarily chosen subspace $B$ of $H^1(K \otimes \Qp, V)$, then a straightforward generalization of Theorem 5.2.15 of \emph{op.cit.} shows that
    \[ \chi(T, \mathcal{F}) = \dim_E(V^-) + \dim_{E} H^0(K \otimes \Qp, V^*(1)) - \dim_{E} \left(\frac{H^1(K \otimes \Qp, V)}{B}\right),\]
    where $V^-$ is the minus eigenspace for complex conjugation acting on $\Ind_K^{\QQ} V$. In our situation, we have taken $B = H^1_f(K \otimes \Qp, V)$, which has dimension 3; and since $K$ is totally complex, $\dim_E(V^-) = \tfrac 1 2 [K : \QQ] \dim_E(V) = 2$. Thus we have $\chi(T, \mathcal{F}_{BK}) = 1$, which explains why one should expect ``interesting'' Kolyvagin systems with this local condition at $p$. Theorem \ref{thm:ESmachine} can then be seen as an instance of Theorem 5.2.2 of \emph{op.cit.}, suitably generalized to $K \ne \QQ$.
   \end{remark}

\providecommand{\bysame}{\leavevmode\hbox to3em{\hrulefill}\thinspace}
\renewcommand{\MR}[1]{MR \#\href{http://www.ams.org/mathscinet-getitem?mr=#1}{#1}.}

\end{document}